\documentclass[pdftex]{amsart}
\usepackage[a4paper, left=28mm, right=28mm, top=28mm, bottom=34mm]{geometry}
\usepackage[all]{xy}
\usepackage{comment}
\usepackage{amsmath}
\usepackage{amssymb}
\usepackage{amsthm}
\usepackage{here}
\usepackage[pdftex]{graphicx}
\usepackage{mathrsfs}
\usepackage{mathtools}
\usepackage{mathabx}
\usepackage{scalefnt}
\usepackage{url}
\usepackage{multirow}
\theoremstyle{plain}
\newtheorem{thm}{Theorem}[section]
\newtheorem{lem}[thm]{Lemma}
\newtheorem{cor}[thm]{Corollary}
\newtheorem{prop}[thm]{Proposition}

\theoremstyle{definition}

\newtheorem{rem}[thm]{Remark}

\newtheorem{defn}[thm]{Definition}

\newtheorem{ex}[thm]{Example}

\numberwithin{equation}{section}

\def\F{{\mathbb F}}
\def\Q{{\mathbb Q}}
\def\R{{\mathbb R}}
\def\Z{{\mathbb Z}}
\def\C{{\mathbb C}}
\def\P{{\mathbb P}}
\def\id{\mathop{\mathrm{id}}\nolimits}

\def\Ker{\mathop{\mathrm{Ker}}\nolimits}

\def\GL{\mathop{\mathrm{GL}}\nolimits}

\def\Tr{\mathop{\mathrm{Tr}}\nolimits}
\def\det{\mathop{\mathrm{det}}\nolimits}
\def\dim{\mathop{\mathrm{dim}}\nolimits}
\def\div{\mathop{\mathrm{div}}\nolimits}

\def\Stab{\mathop{\mathrm{Stab}}\nolimits}

\def\L{\mathscr{L}}

\def\F{\mathscr{F}}
\def\W{\mathscr{W}}
\def\OO{\mathscr{O}}
\def\V{\mathscr{V}}

\def\D{\mathscr{D}}

\def\l{\langle}
\def\r{\rangle}

\def\exp{\mathop{\mathrm{exp}}\nolimits}

\def\SO{\mathop{\mathrm{SO}}\nolimits}

\def\W{W}
\def\ZZ{Z}
\def\U{\mathrm{U}}
\def\O{\mathrm{O}}
\def\V{V}
\def\T{T}
\def\CC{C}

\newcommand{\defeq}{\vcentcolon=}

\begin{document}

\title[Irregular cusps of ball quotients]
{Irregular cusps of ball quotients}

\author{Yota Maeda}
\address{Sony Group Corporation, 1-7-1 Konan, Minato-ku, Tokyo, 108-0075, Japan/Department of Mathematics, Faculty of Science, Kyoto University, Kyoto 606-8502, Japan}
\email{y.maeda@math.kyoto-u.ac.jp}

%\date{\today}
%\date{October 3, 2019}

\subjclass[2010]{}

\date{\today}

\maketitle

\begin{abstract}
We study the branch divisors on the boundary of the canonical toroidal compactification of ball quotients.
We show a criterion, the low slope cusp form trick, for proving that ball quotients are of general type.
Moreover, we classify when irregular cusps exist in the case of the discriminant kernel and construct concrete examples for some arithmetic subgroups.
As another direction of study, when a complex ball is embedded into a Hermitian symmetric domain of type IV, we determine when regular or irregular cusps map to regular or irregular cusps studied by Ma.
\end{abstract}

\section{Introduction}
When calculating the order of modular forms on modular curves at cusps, we need to consider whether the cusp is regular or not.
If it is irregular, then the order of the modular forms is defined as half the order determined by its Fourier expansion at the cusp.
More precisely, irregular cusps of modular curves are cusps whose width are strictly smaller than the period for Fourier expansion; this is explained in detail in \cite{DS}.
In the case of orthogonal modular varieties, Ma \cite{irregular} defined and studied irregular cusps.
He classified the structures of discriminant groups for the case of discriminant kernel when irregular cusps may exist on the orthogonal modular varieties and constructed examples.
Finally, he proved the low slope cusp form trick, which is a modification of the low weight cusp form trick \cite[Theorem 1.1]{GHS} when the irregular cusps arise, and used it to show that some orthogonal modular varieties are of general type.

In this paper, we work on ball quotients.
First, we define irregular cusps on them.
Unlike the case of orthogonal modular varieties, in our situation, there may exist branch divisors with branch index 2,3,4 or 6 as explained in section \ref{section:irregular}.
Considering the effects of these cusps, as a main result, we give a sufficient condition for a ball quotient to be of general type in terms of modular forms, called the low slope cusp form trick.
On the other hand, we shall give an example of a ball quotient of non-negative Kodaira dimension in section \ref{section:non_negative}.
This is done by constructing a cusp form, satisfying a weaker condition appearing in this trick.
Second, we consider the relationship between regular/irregular cusps on ball quotients and regular/irregular cusps on orthogonal modular varieties when a Hermitian symmetric domain of type I is embedded into one of type IV.
In this situation, we prove that regular cusps map to regular cusps and determine whether irregular cusps map to regular or irregular cusps.
Third, we classify the structures of the discriminant group when the discriminant kernel may have irregular cusps in section \ref{discriminant_kernel_case} and Appendix \ref{app:A}.
Finally, we construct concrete examples of irregular cusps of any index for any imaginary quadratic field with class number 1 in section \ref{section:examples}.

Before stating our results, we should summarize our settings.
Let $F$ be an imaginary quadratic field and $\OO_F$ be its ring of integers.
Let $(L, \l\ ,\ \r)$ be a Hermitian lattice of signature $(1,n)$ over $\OO_F$ with $n>1$ and $\U(L)$ be the associated unitary group scheme over $\Z$.
Then, the Hermitian symmetric domain associated with the unitary group $\U(L)(\R)\cong\U(1,n)$ is defined as
\[D_L\defeq\{v\in\mathbb{P}(L\otimes_{\OO_F}\C)\mid \l v,v \r>0\}\]
which is an $n$-dimensional complex ball.
For a finite index subgroup $\Gamma\subset\U(L)(\Z)$, we define the \textit{ball quotient}:
\[\F_L(\Gamma)\defeq D_L/\Gamma.\]

On the other hand, we define the associated quadratic lattice $(L_Q, (\ ,\ ))$ over $\Z$ of signature $(2,2n)$, where $L_Q\defeq L$ as a $\Z$-module and $(\ ,\ )\defeq\Tr_{F/\Q}\l\ ,\ \r$.
Let $\D_{L_Q}$ be the Hermitian symmetric domain associated with $\O^+(L_Q)(\R)\cong\O^+(2,2n)$:
\[\D_{L_Q}\defeq\{v\in\mathbb{P}(L_{Q}\otimes_{\Z}\C)\mid (v,v)=0, (v,\overline{v})>0\}^+.\]
Then, we obtain embeddings $\U(L)\hookrightarrow\SO^+(L_Q)$ and $D_L\hookrightarrow \D_{L_Q}$, as was studied in \cite{Hofmann}.

Now, let us introduce the notion of irregular cusps.
Let $I$ be a rank 1 primitive isotropic sublattice of $L$ and $\Gamma(I)_{\Q}$ be the stabilizer of $I\otimes_{\OO_F} F$.
We denote by $\W(I)_{\Q}$ its unipotent part and $\ZZ(I)_{\Q}$ the center of $\W(I)_{\Q}$.
We say $I$ is \textit{irregular with} $($\textit{at least}$)$ \textit{index $2$} if $\ZZ(I)_{\Q}\cap\Gamma\neq\ZZ(I)_{\Q}\cap\l\Gamma,-\id\r$ holds.
We have to consider whether the cusp corresponding to $I$ branches with  higher index or not for $F=\Q(\sqrt{-1})$ or $\Q(\sqrt{-3})$, but for simplicity, we only concern ourselves with this case here.
At irregular cusps, we have to pay attention to the vanishing order of modular forms and related pluricanonical forms.

Here, we shall state our main result, which is a unitary analog of \cite[Theorem 1.1]{GHS} or \cite[Theorem 8.9]{irregular}.

\begin{thm}[Low slope cusp form trick, Theorem \ref{low_slope_trick_unitary}]
Let $F$ be an imaginary quadratic field and $L$ be a Hermitian lattice of signature $(1,n)$ over $\OO_F$.
For a finite index subgroup $\Gamma\subset\U(L)(\Z)$, we assume that there is a non-zero cusp form $\Psi$ of weight $k$ with respect to $\Gamma$ on $D_L$.
In addition, we make the following assumptions.
\begin{enumerate}
    \item $v_R(\Psi)/k>(r_i-1)/(n+1)$ for every irreducible component $R_i$ of the ramification divisors $D_L\to\F_L(\Gamma)$ with ramification index $r_i$.
    \item $v_I(\Psi)/k>1/(n+1)$ for every regular isotropic sublattice $I\subset L$.
    \item $v_I(\Psi)/k>m_I/(n+1)$ for every $\mathrm{(}$semi-$\mathrm{)}$irregular isotropic sublattice $I\subset L$ with index $m_I$.
    \item $n\geq\mathrm{max}_{i,I}\{r_i-2,m_I-1\}$
    \item $\overline{\F_L(\Gamma)}$ has at worst canonical singularities.
\end{enumerate}
Then the ball quotient $\F_L(\Gamma)$ is of general type.
\end{thm}

\begin{rem}
Assumptions (4) and (5) are satisfied if $n\geq 13$ and $d<-3$ by \cite[Theorem 4]{Behrens}.
\end{rem}
Here, $\overline{\F_L(\Gamma)}$ is the canonical toroidal compactification of $\F_L(\Gamma)$.
For the notion of ``semi-irregular'', see section \ref{section:irregular}.

We also consider the relationship between regular/irregular cusps on $D_L$ and regular/irregular cusps on $\D_{L_Q}$.
Note that irregular cusps on $\D_{L_Q}$ have been studied by Ma \cite{irregular}.
Let $\Gamma_O\subset\O^+(L_Q)(\Z)$ be a finite index subgroup and $\Gamma_U\subset\U(L)(\Z)$ be its restriction to the unitary group.
In the following proposition, regular/irregular cusps on $D_L$ (resp. $\D_{L_Q}$) mean regular/irregular cusps with respect to $\Gamma_U$ (resp. $\Gamma_O$).
\begin{prop}
\label{unitary_to_orthogonal}
\begin{enumerate}
    \item For any imaginary quadratic field $F$, regular cusps on $D_L$ map to regular cusps on $\D_{L_Q}$.
    \item For $F\neq\Q(\sqrt{-1}), \Q(\sqrt{-3})$, irregular cusps on $D_L$ map to irregular cusps on $\D_{L_Q}$.
    \item For $F=\Q(\sqrt{-1})$, irregular cusps with index $2$ or $4$ on $D_L$ map to irregular cusps with index $2$ on $\D_{L_Q}$ and semi-irregular cusps with index $2$ on $D_L$ map to regular cusps on $\D_{L_Q}$.
    \item For $F=\Q(\sqrt{-3})$, irregular cusps with index $2$ or $6$ and semi-irregular cusps with index $2$ on $D_L$ map to irregular cusps with index $2$ on $\D_{L_Q}$ and irregular cusps with index $3$ and semi-irregular cusps with index $3$ on $D_L$ map to regular cusps on $\D_{L_Q}$.
\end{enumerate}

\end{prop}

For the case of discriminant kernel, we completely classify discriminant groups when the lattice may have irregular cusps.
\begin{prop}
If $F$ is an imaginary quadratic field of class number $1$, and the discriminant kernel of a unitary group has an irregular cusp, then the discriminant group of an even Hermitian lattice is one of those listed in Appendix  \ref{app:A}.
\end{prop}

\section{0-dimensional cusps}
Let $F\defeq\Q(\sqrt{d})$ be an imaginary quadratic field, where $d$ is a square-free negative integer.
Let $(L,\l\ ,\ \r)$ be a Hermitian lattice over $\OO_F$ of signature $(1,n)$, where $n>1$ and $V\defeq L\otimes_{\OO_F} F$. 
We will consider integral Hermitian lattices in the sense described in \cite{Hofmann}, that is, 
\[\l\ ,\ \r:L\times L\to \kappa\OO_F\]
where 
\[\kappa\defeq\begin{cases}
\frac{1}{\sqrt{d}}&(d\equiv 1\bmod 4)\\
\frac{1}{2\sqrt{d}}&(d\equiv 2,3\bmod 4).
\end{cases}\]
In this paper, Hermitian forms are complex linear in the first argument  and complex conjugate linear in the second argument.
We also define the dual lattice $L^{\vee}$ of $L$:
\[L^{\vee}\defeq\{v\in L\otimes_{\OO_F} F\mid \l v,w\r\in\kappa\OO_F\ \mathrm{for\ any\ }w\in L\}.\]
This lattice contains $L$ as a finite index lattice, so the \textit{discriminant group} $A_L\defeq L^{\vee}/L$ is a finite $\OO_F$-module.
As an important example of an arithmetic group, the \textit{discriminant kernel} $\widetilde{\U}(L)$ is defined by
\[\widetilde{\U}(L)\defeq\{g\in\U(L)(\Z)\mid g\vert_{A_L}=\id\}.\]

Now, let us recall the toroidal compactification of $\F_L(\Gamma)$ and its cusps.
For a rank 1 primitive isotropic sublattice   $I\subset L$, let $\Gamma(I)_{\Q}\defeq\Stab_{\U(L)(\Q)}(I_F)$ be the stabilizer of $I_F\defeq I\otimes_{\OO_F} F$.
Here, we review the structure of $\Gamma(I)_{\Q}$; see \cite{Behrens} and \cite{looijenga_ball} for details.
Let
\[\W(I)_{\Q}\defeq\Ker(\Gamma(I)_{\Q}\to\U(I^{\perp}/I_F)\times\GL(I_{F}))\]
be the unipotent radical of $\Gamma(I)_{\Q}$ and 
\[\ZZ(I)_{\Q}\defeq\Ker(\Gamma(I)_{\Q}\to\GL(I^{\perp}))\]
be its center.
We fix a generator $e$ of $I$.
By \cite{looijenga_ball}, we define 
\[T_{e\otimes v}(z)\defeq z+\l z,e\r v-\l z,v\r e-\frac{1}{2}\l v,v\r\l z,e\r e\]
for $v\in I^{\perp}$ and $z\in V$.
Then, the following properties hold:
\begin{align*}
\left\{
\begin{array}{ll}
T_{e\otimes\mu v} &= T_{\overline{\mu}e\otimes v}\ (\mu\in F) \\
T_{e\otimes\lambda e} &= \id_V\ (\lambda\in\Q)\\
T_{e\otimes v}T_{e\otimes u}&=T_{e\otimes (v+u+\frac{1}{2}\l v,u\r e)}.
\end{array}
\right.
\end{align*}
Thus, it follows that $T_{e\otimes v}$ depends only on $\overline{I}_F\otimes I^{\perp}/(\overline{I}\otimes I)(\Q)$.
Here,
\begin{comment}
\[(\overline{I}\otimes I)(\Q)\defeq\{\lambda e\otimes\lambda e=|\lambda|^2e\otimes e\mid\lambda\in F\}.\]
\end{comment}
\[(\overline{I}\otimes I)(\Q)\defeq\{\lambda (e\otimes e)\mid\lambda\in\Q\}.\]
From the definition of $T_{e\otimes v}$, it follows $T_{e\otimes v}=\id_{I^{\perp}}$ for $e\otimes v\in\overline{I}_{F}\otimes I_{F}$ so that 
\begin{equation}
\label{unipotent_isom}
\begin{array}{ccc}
    \overline{I}_F\otimes I_F/(\overline{I}\otimes I)(\Q)=\sqrt{d}(\overline{I}\otimes I)(\Q) &\cong& \ZZ(I)_{\Q} \\
    \sqrt{d}\lambda (e\otimes e)&\mapsto &T_{\sqrt{d}\lambda (e\otimes e)}.
\end{array}
\end{equation}
More directly, by choosing a basis $\{e,b_1,\dots,b_{n-1},e'\}$ of $V$ such that $\{e,b_1,\dots,b_{n-1}\}$ is a basis of  $I^{\perp}$ and $\l e,e'\r=1$, the Hermitian form is given by 
\[\left(
\begin{array}{c|c|c}
  0 &0&1\\
\hline
0&B&0\\
\hline
1& 0 & 0
\end{array}
\right)\]
for some Hermitian matrix $B$, and the center of $\W(I)_{\Q}$ is given by 
\[\ZZ(I)_{\Q}=\left\{\left(
\begin{array}{c|c|c}
  1&0&\lambda\sqrt{d}\\
\hline
0&I_{n-1}&0\\
\hline
0&0 & 1
\end{array}
\right)\middle|\ \lambda\in \Q\right\}.\]
This gives the isomorphism (\ref{unipotent_isom}) more explicitly, 
\[
\begin{array}{ccc}
    \overline{I}_F\otimes I_F/(\overline{I}\otimes I)(\Q)=\sqrt{d}(\overline{I}\otimes I)(\Q) &\cong& \ZZ(I)_{\Q} \\
    \sqrt{d}\lambda (e\otimes e)&\mapsto &\left(
\begin{array}{c|c|c}
  1&0&2\lambda\sqrt{d}\\
\hline
0&I_{n-1}&0\\
\hline
0&0 & 1
\end{array}
\right).
\end{array}
\]
(See \cite[Lemma 12]{Behrens} for a description.)
Now, $\Gamma(I)_{\Q}$ acts on both sides of the equation.
The natural action on the left-hand side coincides with the adjoint action on the right-hand side.
\[T_{\sqrt{d}\lambda\gamma(e\otimes e)}=\gamma^{-1}T_{\sqrt{d}\lambda(e\otimes e)}\gamma\quad (\gamma\in\Gamma(I)_{\Q}).\]
We also have the following isomorphism,
\[\V(I)_{\Q}\cong\overline{I}_F\otimes I^{\perp}/I_F\]
by \cite{looijenga_ball}.
Here, $\V(I)_{\Q}$ is defined in (\ref{exact3}).
For a finite index subgroup $\Gamma\subset\U(L)(\Z)$, we introduce the following notation from \cite{AMRT} and \cite{irregular}:
\[
\Gamma(I)_{\Z}\defeq\Gamma(I)_{\Q}\cap\Gamma,\ \W(I)_{\Z}\defeq \W(I)_{\Q}\cap\Gamma,\ \ZZ(I)_{\Z}\defeq \ZZ(I)_{\Q}\cap\Gamma
\]
\[
\overline{\Gamma(I)}_{\Z}\defeq\Gamma(I)_{\Z}/\ZZ(I)_{\Z},\ \V(I)_{\Z}\defeq \W(I)_{\Z}/\ZZ(I)_{\Z},\ \Gamma_I\defeq\Gamma(I)_{\Z}/\W(I)_{\Z}
\]
\[
\overline{\Gamma(I)}_{\Q}\defeq\Gamma(I)_{\Q}/\ZZ(I)_{\Z},\ \W(I)_{\Q/\Z}\defeq \W(I)_{\Q}/\ZZ(I)_{\Z},\ \ZZ(I)_{\Q/\Z}\defeq \ZZ(I)_{\Q}/\ZZ(I)_{\Z}.
\]
Now we have the following exact sequences:
\begin{align}
\label{exact1}
    0\rightarrow \V(I)_{\Z}&\rightarrow\overline{\Gamma(I)}_{\Z}\rightarrow\Gamma_I\rightarrow 1\\\label{exact2}
    0\rightarrow \W(I)_{\Q/\Z}&\rightarrow\overline{\Gamma(I)}_{\Q}\rightarrow\U(I^{\perp}/I_{F})\times\GL(I_{F})\\\label{exact3}
    0\rightarrow \ZZ(I)_{\Q/\Z}&\rightarrow \W(I)_{\Q/\Z}\rightarrow \V(I)_{\Q}\rightarrow 0.
\end{align}
Note that  $\ZZ(I)_{\Q/\Z}$ is a torsion subgroup of $\T(I)\defeq \ZZ(I)_{\C}/\ZZ(I)_{\Z}$.
Let $c_I\defeq\P(I\otimes_{\OO_F}\C)$ be the cusp corresponding to $I$.
We need a representation of $D_L$ as a Siegel domain of the third kind.
We define $D(I)\defeq \ZZ(I)_{\C}D_L$.
Then, we obtain the following fibration by \cite{AMRT}:
\[
D(I)\cong \ZZ(I)_{\C}\times \V(I)_{\C}\times c_I  \stackrel{\pi_1}{\rightarrow}  D(I)'\defeq D(I)/\ZZ(I)_{\C} \stackrel{\pi_2}{\rightarrow} c_I.
\]
Moreover, from this fibration, we have
\[D_L=\{(z,u)\in D(I)\mid\Im(z)-h\l u,u\r\in \CC(I)\}\]
for a cone $\CC(I)$ in $\ZZ(I)_{\R}$ and some real-bilinear quadratic form $h:\C^{n-1}\times\C^{n-1}\to \ZZ(I)_{\R}$.
Accordingly, we have 
\[\mathcal{X}(I)\defeq D/\ZZ(I)_{\Z}\subset D(I)/\ZZ(I)_{\Z}\stackrel{\overline{\pi}_1}{\rightarrow}D(I)'.\]
Here, the quotient fiber bundle $\overline{\pi}_1$ is a principal fiber bundle under the algebraic torus $\T(I)\defeq \ZZ(I)_{\C}/\ZZ(I)_{\Z}$.
Since $\dim_{\R}(\ZZ(I)_{\R})=1$, there exists a natural toric embedding $\T(I)\hookrightarrow \overline{\T(I)}$.
In accordance with \cite{AMRT}, we define $\overline{\mathcal{X}(I)}$ as the interior of closure of $\mathcal{X}(I)$ in $\mathcal{X}(I)\times_{\T(I)}\overline{\T(I)}$.

Finally, the toroidal compactification of $\F_L(\Gamma)$ is defined by taking the canonical cone decomposition:
\[\overline{\F_L(\Gamma)}\defeq(D_L\cup\bigcup_{I\subset L}\overline{\mathcal{X}(I)})/\sim\]
where $I$ is a rank 1 primitive isotropic sublattice of $L$ and the equivalence relation is defined in \cite{AMRT}.
\begin{comment}
\begin{rem}
For several compactifications, we have to consider the rational structure of Hermitian spaces.
In this paper, we have an integral structure of $V$ as $L$, so every isotropic line $I_F\subset V$ arises from a rank 1 primitive isotropic lattice $I\defeq I_F\cap \OO_F$.
Hence we work on $I\subset L$.
\end{rem}
\end{comment}
\begin{rem}
We can also construct the Satake-Baily-Borel compactification $\overline{\F_L(\Gamma)}^{\mathrm{SBB}}$ of a ball quotient  $\F_L(\Gamma)$ as follows.
We define the rational completion $D_L^{\mathrm{SBB}}$ as the union of $D_L$ and 0-dimensional cusps:
\[D_L^{\mathrm{SBB}}\defeq D_L\cup\bigcup_{I\subset L} c_I.\]
Here, $I\subset L$ runs over the rank 1 primitive isotropic sublattices.
Now we define $\overline{\F_L(\Gamma)}^{\mathrm{SBB}}\defeq D_L^{\mathrm{SBB}}/\Gamma$.
\end{rem}

\section{Irregular cusps}
\label{section:irregular}
\subsection{Case of $\Q(\sqrt{-1})$}
Throughout this subsection, we assume $F=\Q(\sqrt{-1})$.
Let us define irregular cusps.
\begin{prop}
\label{sp-2-irreg}
The following are equivalent.
\begin{enumerate}
\item $\ZZ(I)_{\Z}=\ZZ(I)_{\Q}\cap\l\Gamma,-\id\r\neq \ZZ(I)_{\Q}\cap\l\Gamma,\sqrt{-1}\id\r$.
\item $-\id\in\Gamma$, $\sqrt{-1}\id\not\in\Gamma$, and  $\sqrt{-1}T_{\sqrt{-1}\lambda (e\otimes e)}\in\Gamma(I)_{\Z}$ for some $\sqrt{-1}\lambda (e\otimes  e)\in\sqrt{-1}(\overline{I}\otimes I)(\Q)$.
\item $-\id\in\Gamma$, $\sqrt{-1}\id\not\in\Gamma$, and there exists an element $\gamma\in\overline{\Gamma(I)}_{\Z}$ of order $4$, acting on $\ZZ(I)_{\Z}$ and $\V(I)_{\C}$ trivially and $\mathcal{X}(I)$ non-trivially, and whose image in $\U(I^{\perp})\times\GL(I_F)$ is $(\sqrt{-1}\id_{I^{\perp}/I_F},\sqrt{-1}\id_{I_{F}})$.
Moreover, the order of this non-trivial action on $\mathcal{X}(I)$ is $2$.
\end{enumerate}
\end{prop}
\begin{proof}
$(1)\Rightarrow (2)$
Since $\sqrt{-1}\id\not\in\Gamma$, there exists an element  $T_{\sqrt{-1}\lambda (e\otimes e)}\in  \ZZ(I)_{\Q}\cap\l\Gamma,\sqrt{-1}\id\r\setminus \ZZ(I)_{\Z}$ for some $\sqrt{-1}\lambda (e\otimes e)\in\sqrt{-1}(\overline{I}\otimes I)(\Q)$.
Now $\l\Gamma,\sqrt{-1}\id\r=\Gamma\sqcup\sqrt{-1}\Gamma$ so that $T_{\sqrt{-1}\lambda (e\otimes e)}\in\sqrt{-1}\Gamma$.
Combining this with the condition $-\id\in\Gamma$, it follows $\sqrt{-1}T_{\sqrt{-1}\lambda (e\otimes e)}\in\Gamma(I)_{\Z}$.
\newline

$(2)\Rightarrow (1)$
Since $-\id\in\Gamma$, we have $\ZZ(I)_{\Z}=\ZZ(I)_{\Q}\cap\l\Gamma,-\id\r$.
On the other hand, $\sqrt{-1}T_{\sqrt{-1}\lambda (e\otimes e)}\in\Gamma(I)_{\Z}$ and $\sqrt{-1}\id\not\in\Gamma$ together shows that $T_{\sqrt{-1}\lambda (e\otimes e)}\in \ZZ(I)_{\Q}\cap\l\Gamma,\sqrt{-1}\id\r\setminus \ZZ(I)_{\Z}$.
\newline

$(2)\Rightarrow (3)$
Let $\gamma\defeq\sqrt{-1}T_{\sqrt{-1}\lambda e\otimes\lambda e}$ be an order 4 element in $\overline{\Gamma(I)}_{\Z}$.
The element $\gamma$ acts on $\overline{I}$ as $-\sqrt{-1}$-times and $I^{\perp}/I$ as $\sqrt{-1}$-times.
Hence, $\gamma$ acts on $\V(I)_{\C}$ trivially.
By definition, $\sqrt{-1}\id$ and $T_{\sqrt{-1}\lambda (e\otimes e)}$ act on $\ZZ(I)_{\Z}$ trivially, so the same holds for $\gamma$.
We also have the image of $\gamma\in\overline{\Gamma(I)_{\Z}}$ in $\U(I^{\perp})\times\GL(I_F)$ is  $(\sqrt{-1}\id_{I^{\perp}/I_F},\sqrt{-1}\id_{I_F})$.

On the other hand, under the assumption $\sqrt{-1}\not\in\Gamma$, it follows $T_{\sqrt{-1}\lambda (e\otimes e)}\not\in \ZZ(I)_{\Z}$.
This means that $\gamma$ acts on $\mathcal{X}(I)$ non-trivially.
Note that $\ZZ(I)_{\Q}$ acts on $\mathcal{X}(I)\subset T(I)\defeq \ZZ(I)_{\C}/\ZZ(I)_{\Z}$ as a  translation, so the above action is a non-trivial translation.
\newline

$(3)\Rightarrow (2)$
From (\ref{exact2}), we have $\gamma=(\sqrt{-1}\id_{I^{\perp}/I_{F}},\sqrt{-1}\id_{I_{F}},\alpha)$ for some $\alpha\in \W(I)_{\Q/\Z}$.
Since $\gamma$ acts on $\V(I)_{\C}$ trivially, it follows that the image of $\alpha$ in $\V(I)_{\Q}$ is 0 in (\ref{exact3}), so $\alpha\in \ZZ(I)_{\Q/\Z}$.
Hence, $\gamma=(\sqrt{-1}\id_{I^{\perp}/I_{F}},\sqrt{-1}\id_{I_{F}},T_{\sqrt{-1}\lambda (e\otimes e)})$ for some $\sqrt{-1}\lambda (e\otimes  e)\in\sqrt{-1}(\overline{I}\otimes I)(\Q)$.
Now, we have $\sqrt{-1}\id_{L}=(\sqrt{-1}\id_{I^{\perp}/I_{F}},\sqrt{-1}\id_{I_{F}},0)$, so combining this with $\gamma=(\sqrt{-1}\id_{I^{\perp}/I_{F}},\sqrt{-1}\id_{I_{F}},T_{\sqrt{-1}\lambda (e\otimes e)})$, it follows $\sqrt{-1}\gamma=-T_{\sqrt{-1}\lambda( e\otimes e)}\in\Gamma$.
Since we have assumed $-\id\in\Gamma$ so that $\sqrt{-1}T_{\sqrt{-1}\lambda (e\otimes  e)}\in\Gamma$.

\end{proof}
Geometrically, the existence of such a cusp corresponds to the existence of a branch divisor on the boundary of a  ball quotient with branch index 2.
We can show the following propositions in the same way as Proposition \ref{sp-2-irreg}.

\begin{defn}
We say that $I$ is \textit{semi-irregular with index $2$} if the conditions in Proposition \ref{sp-2-irreg}  are satisfied.
Here, we define $\ZZ(I)'_{\Z}\defeq \ZZ(I)_{\Q}\cap\l\Gamma,\sqrt{-1}\id\r$ and $\Gamma(I)'_{\Z}\defeq\l\Gamma(I)_{\Z},\sqrt{-1}\id\r/\l\sqrt{-1}\id\r$.
\end{defn}

Now, let us treat the index 4 case.

\begin{prop}
\label{4-irreg}
The following statements are equivalent.
\begin{enumerate}
\item $\ZZ(I)_{\Z}\neq \ZZ(I)_{\Q}\cap\l\Gamma,-\id\r\neq \ZZ(I)_{\Q}\cap\l\Gamma,\sqrt{-1}\id\r$,  that is, all three are different.
\item $-\id,\sqrt{-1}\id\not\in\Gamma$, and  $-\sqrt{-1}T_{\sqrt{-1}\lambda (e\otimes e)}\in\Gamma(I)_{\Z}$ for some $\sqrt{-1}\lambda (e\otimes  e)\in\sqrt{-1}(\overline{I}\otimes I)(\Q)$.
\item $\sqrt{-1}\id\not\in\Gamma$, and there exists an element $\gamma\in\overline{\Gamma(I)}_{\Z}$ of order $4$ acting on $\ZZ(I)_{\Z}$ and $\V(I)_{\C}$ trivially and $\mathcal{X}(I)$ non-trivially, and whose image in $\U(I^{\perp})\times\GL(I_F)$ is $(\sqrt{-1}\id_{I^{\perp}/I_{F}},\sqrt{-1}\id_{I_{F}})$.
Moreover, the order of this non-trivial action on $\mathcal{X}(I)$ is $4$.
\end{enumerate}
\end{prop}
\begin{proof}
This can be proven in the same way as Proposition \ref{sp-2-irreg}.
\end{proof}

\begin{defn}
We say that $I$ is \textit{irregular with index $4$} if the conditions in Proposition \ref{4-irreg}   are satisfied.
Here, we define $\ZZ(I)'_{\Z}\defeq \ZZ(I)_{\Q}\cap\l\Gamma,\sqrt{-1}\id\r$ and $\Gamma(I)'_{\Z}\defeq\l\Gamma(I)_{\Z},\sqrt{-1}\id\r/\l\sqrt{-1}\id\r$.
\end{defn}

\subsection{Case of $\Q(\sqrt{-3})$}
Throughout this subsection, we assume $F=\Q(\sqrt{-3})$.
Let $\omega$ be a primitive root of unity.

\begin{prop}
\label{sp-2-irreg'}
The following statements are equivalent.
\begin{enumerate}
\item $\ZZ(I)_{\Z}=\ZZ(I)_{\Q}\cap\l\Gamma,\omega\id\r\neq \ZZ(I)_{\Q}\cap\l\Gamma,-\id\r$.
\item $\omega\id\in\Gamma$, $-\id\not\in\Gamma$, and  $-T_{\sqrt{-3}\lambda (e\otimes e)}\in\Gamma(I)_{\Z}$ for some $\sqrt{-3}\lambda (e\otimes  e)\in\sqrt{-3}(\overline{I}\otimes I)(\Q)$.
\item $\omega\id\in\Gamma$, $-\id\not\in\Gamma$, and there exists an element $\gamma\in\overline{\Gamma(I)}_{\Z}$ of order $6$, acting on $\ZZ(I)_{\Z}$ and $\V(I)_{\C}$ trivially and $\mathcal{X}(I)$ non-trivially, and whose image in $\U(I^{\perp})\times\GL(I_F)$ is $(-\id_{I^{\perp}/I_{F}},-\id_{I_{F}})$.
Moreover, the order of this non-trivial action on $\mathcal{X}(I)$ is $2$.
\end{enumerate}
\end{prop}
\begin{proof}
This can be proven in the same way as Proposition \ref{sp-2-irreg}.
\end{proof}

\begin{defn}
We say that $I$ is \textit{semi-irregular with index $2$} if the conditions in Proposition \ref{sp-2-irreg'}  are satisfied.
Here, we define $\ZZ(I)'_{\Z}\defeq \ZZ(I)_{\Q}\cap\l\Gamma,\omega\id\r$ and $\Gamma(I)'_{\Z}\defeq\l\Gamma(I)_{\Z},\omega\id\r/\l\omega\id\r$.
\end{defn}

\begin{prop}
\label{sp-3-irreg}
The following statements are equivalent.
\begin{enumerate}
\item $\ZZ(I)_{\Z}=\ZZ(I)_{\Q}\cap\l\Gamma,-\id\r\neq \ZZ(I)_{\Q}\cap\l\Gamma,-\omega\id\r$.
\item $-\id\in\Gamma$, $\omega\id\not\in\Gamma$, and  $-T_{\sqrt{-3}\lambda (e\otimes e)}\in\Gamma(I)_{\Z}$ for some $\sqrt{-3}\lambda (e\otimes  e)\in\sqrt{-3}(\overline{I}\otimes I)(\Q)$.
\item $-\id\in\Gamma$, $\omega\id\not\in\Gamma$, and there exists an element $\gamma\in\overline{\Gamma(I)}_{\Z}$ of order $6$, acting on $\ZZ(I)_{\Z}$ and $\V(I)_{\C}$ trivially and $\mathcal{X}(I)$ non-trivially, and whose image in $\U(I^{\perp})\times\GL(I_F)$ is $(\omega\id_{I^{\perp}/I_{F}},\omega\id_{I_{F}})$.
Moreover, the order of this non-trivial action on $\mathcal{X}(I)$ is $3$.
\end{enumerate}
\end{prop}
\begin{proof}
This can be proven in the same way as Proposition \ref{sp-2-irreg}.
\end{proof}

\begin{defn}
We say that $I$ is \textit{semi-irregular with index $3$} if the conditions in Proposition \ref{sp-2-irreg'}  are satisfied.
Here, we define $\ZZ(I)'_{\Z}\defeq \ZZ(I)_{\Q}\cap\l\Gamma,\omega\id\r$ and $\Gamma(I)'_{\Z}\defeq\l\Gamma(I)_{\Z},\omega\id\r/\l\omega\id\r$.
\end{defn}

\begin{prop}
\label{3-irreg}
The following statements are equivalent.
\begin{enumerate}
\item $\ZZ(I)_{\Z}=\ZZ(I)_{\Q}\cap\l\Gamma,-\id\r\neq \ZZ(I)_{\Q}\cap\l\Gamma,\omega\id\r$.
\item $-\id\in\Gamma$, $\omega\id\not\in\Gamma$, and  $\omega T_{\sqrt{-3}\lambda (e\otimes e)}\in\Gamma(I)_{\Z}$ for some $\sqrt{-3}\lambda (e\otimes  e)\in\sqrt{-3}(\overline{I}\otimes I)(\Q)$.
\item $-\id\in\Gamma$, $\omega\id\not\in\Gamma$, and there exists an element $\gamma\in\overline{\Gamma(I)}_{\Z}$ of order $3$, acting on $\ZZ(I)_{\Z}$ and $\V(I)_{\C}$ trivially and $\mathcal{X}(I)$ non-trivially, and whose image in $\U(I^{\perp})\times\GL(I_F)$ is $(\omega\id_{I^{\perp}/I_{F}},\omega\id_{I_{F}})$.
Moreover, the order of this non-trivial action on $\mathcal{X}(I)$ is $3$.
\end{enumerate}
\end{prop}
\begin{proof}
This can be proven in the same way as Proposition \ref{sp-2-irreg}.
\end{proof}

\begin{defn}
We say that $I$ is \textit{irregular with index $3$} if the conditions in Proposition \ref{3-irreg}  are satisfied.
Here, we define $\ZZ(I)'_{\Z}\defeq \ZZ(I)_{\Q}\cap\l\Gamma,\omega\id\r$ and $\Gamma(I)'_{\Z}\defeq\l\Gamma(I)_{\Z},\omega\id\r/\l\omega\id\r$.
\end{defn}

\begin{prop}
\label{6-irreg}
The following statements are equivalent.
\begin{enumerate}
\item $\ZZ(I)_{\Z}\neq \ZZ(I)_{\Q}\cap\l\Gamma,-\id\r\neq \ZZ(I)_{\Q}\cap\l\Gamma,\omega\r$, that is, all three are different.
\item $-\id,\omega\id\not\in\Gamma$, and  $-\omega T_{\sqrt{-3}\lambda (e\otimes  e)}\in\Gamma(I)_{\Z}$ for some $\sqrt{-3}\lambda (e\otimes  e)\in\sqrt{-3}(\overline{I}\otimes I)(\Q)$.
\item $-\id,\omega\id\not\in\Gamma$, and there exists an element $\gamma\in\overline{\Gamma(I)}_{\Z}$ of order $6$, acting on $\ZZ(I)_{\Z}$ and $\V(I)_{\C}$ trivially and $\mathcal{X}(I)$ non-trivially, and whose image in $\U(I^{\perp})\times\GL(I_F)$ is $(-\omega\id_{I^{\perp}/I_{F}},-\omega\id_{I_{F}})$.
Moreover, the order of this non-trivial action on $\mathcal{X}(I)$ is $6$.
\end{enumerate}
\end{prop}
\begin{proof}
This can be proven in the same way as Proposition \ref{sp-2-irreg}.
\end{proof}

\begin{defn}
We say that $I$ is \textit{irregular with index $6$} if the conditions in Proposition \ref{6-irreg}  are satisfied.
Here, we define $\ZZ(I)'_{\Z}\defeq \ZZ(I)_{\Q}\cap\l\Gamma,-\id,\omega\id\r$ and $\Gamma(I)'_{\Z}\defeq\l\Gamma(I)_{\Z},-\id, \omega\id\r/\l-\id,\omega\id\r$.
\end{defn}

\subsection{Other cases}
Let $F$ be any imaginary quadratic field.
\begin{prop}
\label{2-irreg}
The following statements are equivalent.
\begin{enumerate}
\item $\ZZ(I)_{\Z}\neq \ZZ(I)_{\Q}\cap\l\Gamma,-\id\r$.
\item $-\id\not\in\Gamma$, and  $-T_{\sqrt{d}\lambda (e\otimes e)}\in\Gamma(I)_{\Z}$ for some $\sqrt{d}\lambda (e\otimes e)\in\sqrt{d}(\overline{I}\otimes I)(\Q)$.
\item $-\id\not\in\Gamma$, and there exists an element $\gamma\in\overline{\Gamma(I)}_{\Z}$ of order $2$, acting on $\ZZ(I)_{\Z}$ and $\V(I)_{\C}$ trivially and $\mathcal{X}(I)$ non-trivially, and whose image in $\U(I^{\perp})\times\GL(I_F)$ is $(-\id_{I^{\perp}/I_{F}},-\id_{I_{F}})$.
Moreover, the order of this non-trivial action on $\mathcal{X}(I)$ is $2$.
\end{enumerate}
\end{prop}
\begin{proof}
This can be proven in the same way as Proposition \ref{sp-2-irreg}.
\end{proof}

\begin{defn} 
We say that $I$ is \textit{irregular with index $2$} if the following holds.
If $F\neq\Q(\sqrt{-1}),\Q(\sqrt{-3})$, then the conditions in  Proposition \ref{2-irreg} are satisfied.
If $F=\Q(\sqrt{-1})$, then the conditions in  Proposition \ref{2-irreg} are satisfied and the conditions in  Proposition \ref{4-irreg} are not satisfied.
If $F=\Q(\sqrt{-3})$, then the conditions in  Proposition \ref{2-irreg} are satisfied and the conditions in Proposition \ref{sp-2-irreg'} and Proposition \ref{6-irreg} are not satisfied.
In these cases, we define $\ZZ(I)'_{\Z}\defeq \ZZ(I)_{\Q}\cap\l\Gamma,-\id\r$ and $\Gamma(I)'_{\Z}\defeq\l\Gamma(I)_{\Z},-\id\r/\l-\id\r$.
\end{defn}

\begin{defn}
We say that $I$ is \textit{regular} if $I$ is not irregular or semi-irregular in the sense of the above definitions.
\end{defn}

\subsection{Relation with irregular cusps on orthogonal modular varieties}
Now, let us give another description of regular or irregular cusps.
We define
\[\ZZ(I)^{\star}_{\Z}\defeq\begin{cases}
(\{\pm1,\pm\sqrt{-1}\}\ZZ(I)_{\Q})\cap\Gamma & (F=\Q(\sqrt{-1}))\\
(\{\pm1,\pm\omega,\pm\omega^2\}\ZZ(I)_{\Q})\cap\Gamma & (F=\Q(\sqrt{-3}))\\
(\{\pm1\}\ZZ(I)_{\Q})\cap\Gamma & (F\neq\Q(\sqrt{-1}, \Q(\sqrt{-3}))).\\
\end{cases}\]
We can classify irregular cusps according to the structure of $\ZZ(I)_{\Z}^{\star}/\ZZ(I)_{\Z}$.

For $F=\Q(\sqrt{-1})$, 
\[\ZZ(I)_{\Z}^{\star}/\ZZ(I)_{\Z}\cong\begin{cases}
1 & (\mathrm{type}\ R_1) \\
\l-\id\r\cong\Z/2\Z & (\mathrm{type}\ R_2) \\
\l\sqrt{-1}\id\r\cong\Z/4\Z & (\mathrm{type}\ R_4) \\
\l-T_{\sqrt{-1}\lambda(e\otimes e)}\r\cong\Z/2\Z & (\mathrm{type}\ I_2) \\
\l-\id,-\sqrt{-1}T_{\sqrt{-1}\lambda(e\otimes e)}\r\cong\Z/2\Z\times\Z/2\Z & (\mathrm{type}\ SI_{2}) \\
\l-\sqrt{-1}T_{\sqrt{-1}\lambda(e\otimes e)}\r\cong\Z/4\Z & (\mathrm{type}\ I_4). \\
\end{cases}\]

For $F=\Q(\sqrt{-3})$, 
\[\ZZ(I)_{\Z}^{\star}/\ZZ(I)_{\Z}\cong\begin{cases}
1 & (\mathrm{type}\ R_1) \\
\l-\id\r\cong\Z/2\Z & (\mathrm{type}\ R_2) \\
\l\omega\id\r\cong\Z/3\Z & (\mathrm{type}\ R_3) \\
\l-\omega\id\r\cong\Z/6\Z & (\mathrm{type}\ R_6) \\
\l-T_{\sqrt{-3}\lambda(e\otimes e)}\r\cong\Z/2\Z & (\mathrm{type}\ I_2) \\
\l-\omega,-T_{\sqrt{-3}\lambda(e\otimes e)}\r\cong\Z/3\Z\times\Z/2\Z\cong\Z/6\Z & (\mathrm{type}\ SI_{2}) \\
\l\omega T_{\sqrt{-3}\lambda(e\otimes e)}\r\cong\Z/3\Z & (\mathrm{type}\ I_3) \\
\l-\id, \omega T_{\sqrt{-3}\lambda(e\otimes e)}\r\cong\Z/2\Z\times\Z/3\Z\cong\Z/6\Z & (\mathrm{type}\ SI_{3}) \\
\l-\omega T_{\sqrt{-3}\lambda(e\otimes e)}\r\cong\Z/6\Z & (\mathrm{type}\ I_6). \\
\end{cases}\]

For $F\neq\Q(\sqrt{-1}), \Q(\sqrt{-3})$, 
\[\ZZ(I)_{\Z}^{\star}/\ZZ(I)_{\Z}\cong\begin{cases}
1 & (\mathrm{type}\ R_1) \\
\l-\id\r\cong\Z/2\Z & (\mathrm{type}\ R_2) \\
\l-T_{\sqrt{d}\lambda(e\otimes e)}\r\cong\Z/2\Z & (\mathrm{type}\ I_2). \\
\end{cases}\]
Here, type $R_{\star}$ corresponds to regular cusps, and type $I_{\star}$  (resp.  $SI_{\star}$) corresponds to irregular (resp. semi-irregular) cusps with index $\star$.

Now we will explicitly show how the type of cusps varies when arithmetic subgroups change, and consider the relationship between unitary cusps and orthogonal cusps.
Figures 1, 2, 3 show whether the cusps with respect to finite index subgroups of $\U(L)(\Z)$ are regular or irregular according to inclusions.
We fix an irregular cusp $I$.
For a finite index subgroup $\Gamma\subset\U(L)(\Z)$, these figures represent the type  candidates of another finite index subgroup $\Gamma'\subset\U(L)(\Z)$ having the inclusion relationship with $\Gamma$.
If $\Gamma\subset\Gamma'$ and $\Gamma$ is type $X$, then $\Gamma'$ is type located above $X$ in the  figures, and if $\Gamma'\subset\Gamma$, then $\Gamma'$ is type located below $X$ in the  figures.
For example, in Figure 1, for $F=\Q(\sqrt{-1})$, let $\Gamma$ be type $R_2$.
Then $\Gamma'\supset\Gamma$ is type $R_2$, $SI_2$ or $R_4$.
On the other hand if $\Gamma'\subset\Gamma$, then $\Gamma'$ is type $R_2$, $I_2$ or $R_1$.
Circle nodes mean regular cusps and diamond nodes mean irregular cusps.

\begin{figure}[H]
  \includegraphics[width=1cm]{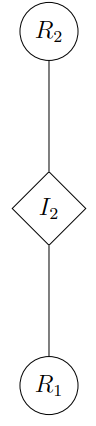}
          \caption{$F\neq\Q(\sqrt{-1}),\Q(\sqrt{-3})$}
\end{figure}
\begin{figure}[H]
  \includegraphics[width=4cm]{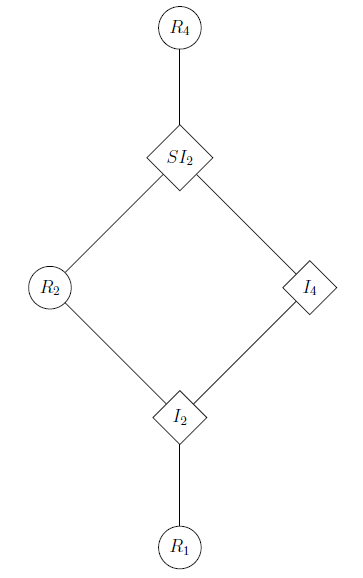}
          \caption{$F=\Q(\sqrt{-1})$}
\end{figure}
\begin{figure}[H]
  \includegraphics[width=10cm]{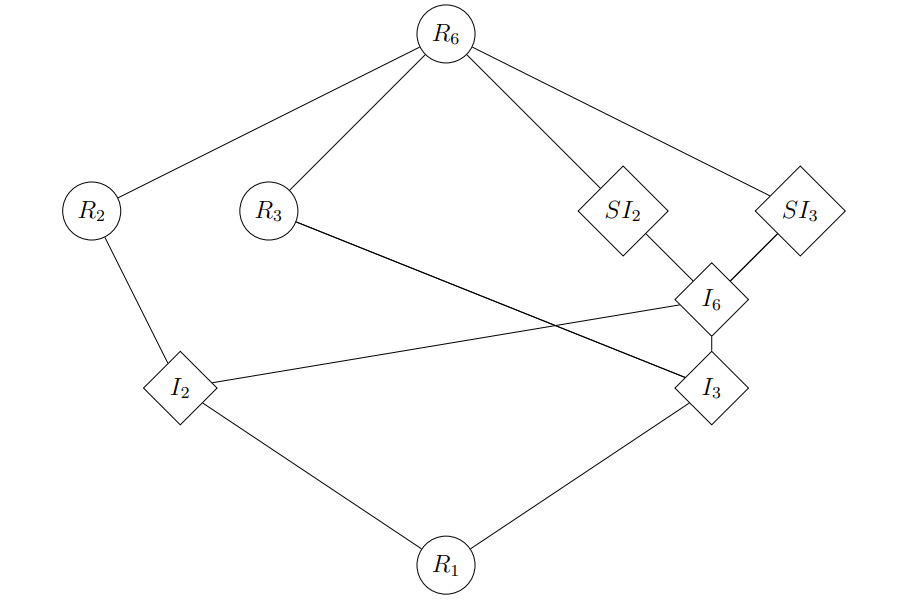}
          \caption{$F=\Q(\sqrt{-3})$}
\end{figure}

Next, let us discuss the relationship between regular/irregular cusps on ball quotients and regular/irregular cusps on orthogonal modular varieties, as studied in \cite{irregular}.
We will borrow a notion of Hermitian forms; i.e., $(L,\l\ ,\ \r)$ denotes a Hermitian lattice of signature $(1,n)$ over $\OO_F$.
We can embed unitary Hermitian symmetric domains (type I) into orthogonal Hermitian symmetric domains (type IV).
In regard to the following discussion on orthogonal modular varieties, the reader may find it informative to consult \cite{Hofmann, Maeda2, MaedaOdaka}.

Let $(L_Q,(\ ,\ ))$ be the associated quadratic lattice over $\Z$ of signature $(2,2n)$, i.e., $L_Q\defeq L$ as a $\Z$-module and $(\ ,\ )\defeq\Tr_{F/\Q}\l\ ,\ \r$.
The associated orthogonal Hermitian symmetric domain is defined by
\[\D_{L_Q}\defeq\{v\in\mathbb{P}(L_Q\otimes\C)\mid (v,v)=0, (v,\overline{v})>0\}^+.\]
Then, we obtain the following embedding:
\begin{equation}
\label{iota}
    \iota:D_L\hookrightarrow \D_{L_Q}.
\end{equation}
By abuse of notation, we also denote by $\iota:\U(L)\hookrightarrow\O^+(L_Q)$.
In this embedding, we identify the unitary group $\U(L)$ with a subgroup of $\O^+(L_Q)$.
Specifically, we get
\[\U(V)=\{\gamma\in\O^+((L_Q)_{\Q})\mid j_d\gamma j_d=d\gamma\}\]
where $j_d\in\O^+((L_Q)_{\Q})$ satisfies $j_d^2=d\id_{L_Q}$.
Explicitly, 
\[j_d\defeq\left(
\begin{array}{cccc}
\left(
\begin{array}{cc}
0&d\\
1&0
\end{array}
\right)&0&0\\
0&\ddots & 0\\
0 & 0 &\left(
\begin{array}{cc}
0&d\\
1&0
\end{array}
\right)
\end{array}
\right).\]

We are concerned whether the image of regular/irregular cusps on ball quotients by $(\ref{iota})$ are regular or irregular on orthogonal modular varieties.
By \cite[Proposition 2]{Hofmann}, a 0-dimensional cusp on $D_L$, corresponding to a rank 1 primitive isotropic sublattice $I\subset L$ maps to a 1-dimensional cusp on $\D_{L_Q}$, corresponding to the rank 2 primitive isotropic sublattice $I_Q\subset L_Q$ spanned by $I$ and $\sqrt{d}I$ (or ($1+\sqrt{d})/2I$ for $d\equiv 1\bmod 4$).
Ma studied irregular cusps on orthogonal modular varieties; here, we will review some of his  results.
In orthogonal cases, only 2-ramifications may occur; they are classified as follows:

\[\ZZ(I_Q)_{\Z}^{\star}/\ZZ(I_Q)_{\Z}\cong\begin{cases}
1 & (\mathrm{type}\ (R_1)_O) \\
\l-\id\r\cong\Z/2\Z & (\mathrm{type}\ (R_2)_O) \\
\l-T_{\sqrt{d}\lambda(e\otimes e)}\r\cong\Z/2\Z & (\mathrm{type}\ (I_2)_O) \\
\end{cases}\]
where $\ZZ(I_Q)_{\Z}$ is the intersection of the center of the unipotent part of the stabilizer of $I_Q$ in $\O^+((L_Q)_{\Q})$ and a finite index subgroup $\Gamma_O\subset\O^+(L_Q)(\Z)$ as in our unitary case.
Type $(R_1)_O$ and $(R_2)_O$ (resp. $(I_2)_O$) means that $I_Q$ is regular (resp. irregular with index 2) in $\D_{L_Q}/\Gamma_O$.
Note that the image of $\ZZ(I)_{\Q}$ is precisely $\ZZ(I_Q)_{\Q}$ and the image of the discriminant kernel in the unitary group is a subgroup of the discriminant kernel in the orthogonal group.
By \cite[Corollary 3,6]{irregular}, we obtain Figure 4 in the orthogonal case.

\begin{figure}[H]
  \includegraphics[width=1cm]{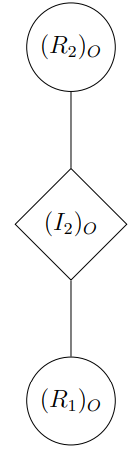}
          \caption{Orthogonal case}
\end{figure}

Now, let us study the image of regular/irregular cusps on orthogonal modular varieties.
Refer to Figures 5, 6, 7.
By \cite[Lemma 2.5]{GHS}, for a 1-dimensional cusp $J\subset L_Q$, the center of the unipotent part of its stabilizer in $\O^+((L_Q)_{\Q})$ is described as
\[\ZZ(J)_{\Q}=\left\{\left(
\begin{array}{cccc}
  I_2&0&\left(
\begin{array}{cc}
0&e\lambda\\
\lambda&0
\end{array}
\right)\\
0&I_{2n-2} & 0\\
0 & 0 & I_2
\end{array}
\right)\middle|\ \lambda\in \Q\right\}\]
for some $e\in\Q$.
For a 2-dimensional $\Q$-isotropic subspace $J_{\Q}\subset (L_Q)_{\Q}$, if we consider it to be a subset of $V$, it defines an $F$-subspace of $V$ if and only if $e=d$.
In that case, the corresponding subspace $I_F$ is a 1-dimensional $F$-isotropic subspace of $V$ and hence corresponds to a 0-dimensional cusp.
This shows that when $e=d$, $\iota(\ZZ(I)_{\Q})=\ZZ(J)_{\Q}$.
We also have $\iota(-\id)=-\id$, $\iota(\sqrt{-1}\id)=j_{-1}$ and 
\[\iota(\omega\id)=\left(
\begin{array}{cccc}
\left(
\begin{array}{cc}
0&1\\
-1&-1
\end{array}
\right)&0&0\\
0&\ddots & 0\\
0 & 0 &\left(
\begin{array}{cc}
0&1\\
-1&-1
\end{array}
\right)
\end{array}
\right).\]
In this situation, consider the following problem.
Let $J\subset L_Q$ be a 1-dimensional cusp and $e=d$ as above.
Let $I\subset L$ be the corresponding 0-dimensional cusp.
Note that $\iota(\ZZ(I)_{\Q})=\ZZ(J)_{\Q}$ holds.
We assume $J$ is a regular or an irregular cusp in the sense of \cite[Definition 6.2]{irregular} with respect to a finite index subgroup $\Gamma_{O}\subset\O^+(L_Q)(\Z)$.
We shall determine whether the corresponding cusp $I$ is regular or irregular in the sense of the above definitions with respect to $\Gamma_U\defeq\iota^{-1}(\Gamma_O)$.

If $J$ is irregular, then $\Gamma_O$ is type $(I_2)_O$.
In this case, since $-\id\not\in\Gamma_O$, we have $-\id\not\in\Gamma_U$; moreover, from the fact $\iota(\ZZ(I)_{\Q})=\ZZ(J)_{\Q}$, it follows that $I$ is irregular and $\Gamma_U$ is type $I_2$, $I_4$, $SI_2$ or $I_6$.
On the other hand, if $J$ is irregular, then $\Gamma_O$ is type $(R_1)_O$ or $(R_2)_O$.
In the first case, since $-\id\in\Gamma_O$, it follows that $-\id\in\Gamma_U$, so we have that  $\Gamma_U$ is type $R_1$, $R_3$, or $I_3$.
In the second case, since $-\id\not\in\Gamma_U$, it follows that $\Gamma_U$ is type $R_2$, $R_4$, $SI_2$ or $SI_3$.

In the following figures, star nodes mean that regular cusps in unitary groups become irregular cusps in orthogonal groups.
These figures show what the type of $\Gamma_O\subset\O^+(L_Q)(\Z)$ is when   $\Gamma_{U}\subset\U(L)(\Z)$ is a certain type.
For example, for $F=\Q(\sqrt{-1})$, if $\Gamma_U\subset\U(L)(\Z)$ is type $R_4$, then the corresponding 1-dimensional cusp is type $(R_2)_O$.
Indeed, regular cusps on $D_L$ map to regular cusps on $\D_{L_Q}$. 
On the other hand, for $F=\Q(\sqrt{-3})$, if $\Gamma\subset\U(L)(\Z)$ is type $SI_3$, i.e, semi-irregular with index 3, then the corresponding 1-dimensional cusp is regular (type $(R_2)_O$).

\begin{figure}[H]
  \includegraphics[width=8cm]{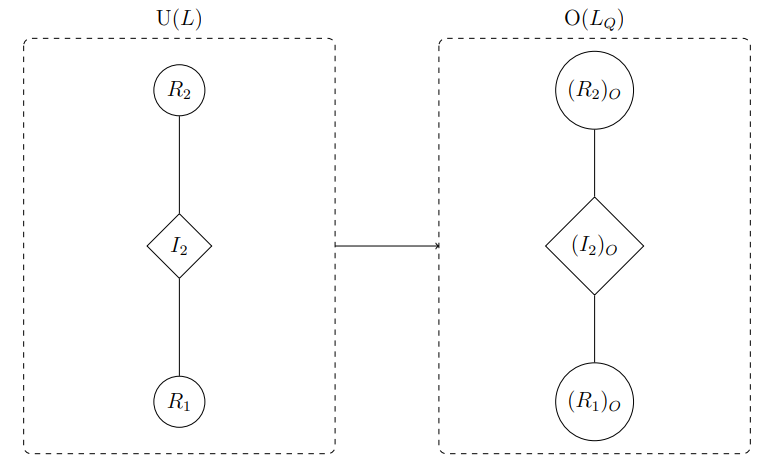}
          \caption{Relationship between unitary and orthogonal for $F\neq\Q(\sqrt{-1}),\Q(\sqrt{-3})$}
\end{figure}
\begin{figure}[H]
  \includegraphics[width=8cm]{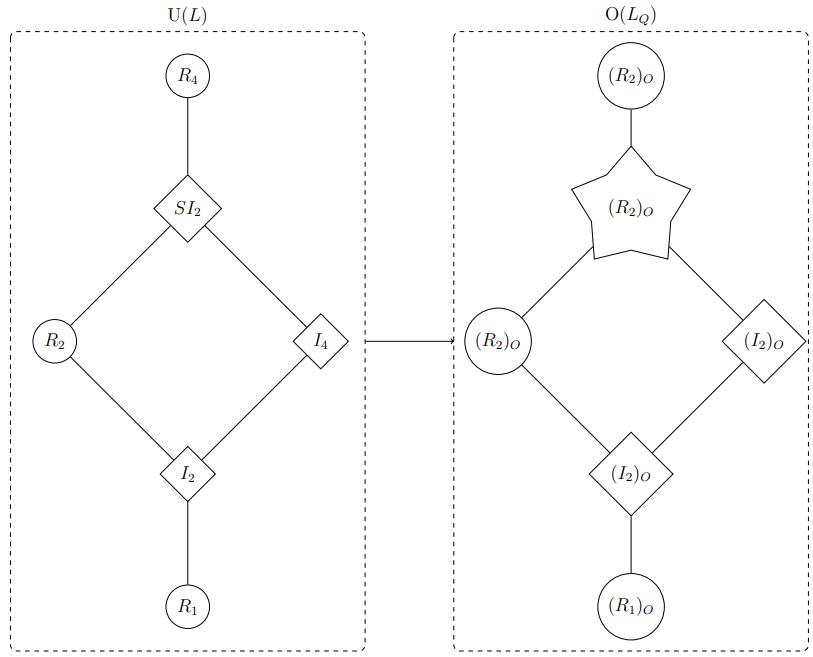}
          \caption{Relationship between unitary and orthogonal for $F=\Q(\sqrt{-1})$}
\end{figure}
\begin{figure}[H]
  \includegraphics[width=8cm]{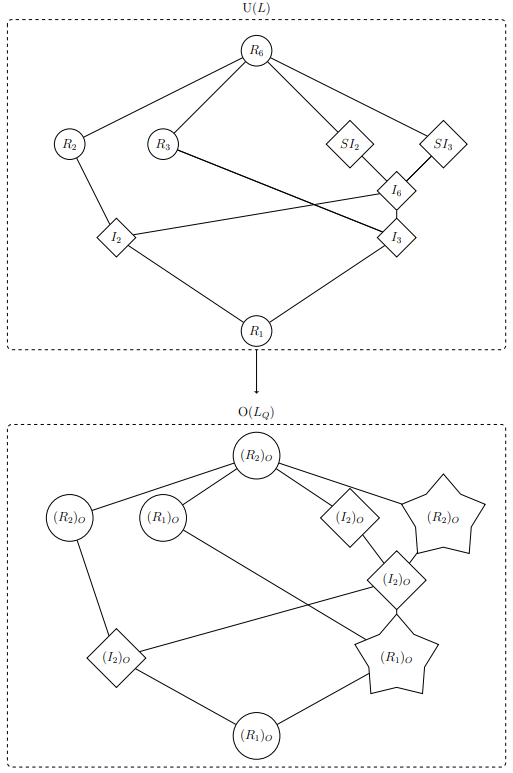}
          \caption{Relationship between unitary and orthogonal for $F=\Q(\sqrt{-3})$}
\end{figure}

From the figures 5, 6, 7, we obtain the following proposition.
Let $\Gamma_O\subset\O^+(L_Q)(\Z)$ be a finite index subgroup and $\Gamma_U\subset\U(L)(\Z)$ be its restriction.
Here, regular/irregular cusps on $D_L$ (resp. $\D_{L_Q}$) mean regular/irregular cusps with respect to $\Gamma_U$ (resp. $\Gamma_O$).
\begin{prop}
\label{unitary_to_orthogonal}
\begin{enumerate}
    \item For any imaginary quadratic field $F$, regular cusps on $D_L$  map to regular cusps on $\D_{L_Q}$.
    \item For $F\neq\Q(\sqrt{-1}), \Q(\sqrt{-3})$, irregular cusps on $D_L$ map to irregular cusps on $\D_{L_Q}$.
    \item For $F=\Q(\sqrt{-1})$, irregular cusps with index $2$ or $4$ on $D_L$ map to irregular cusps with index $2$ on $\D_{L_Q}$, and semi-irregular cusps with index $2$ on $D_L$ map to regular cusps on $\D_{L_Q}$.
    \item For $F=\Q(\sqrt{-3})$, irregular cusps with index $2$ or $6$ and semi-irregular cusps with index 2 on $D_L$ map to irregular cusps with index $2$ on $\D_{L_Q}$, and irregular cusps with index 3 and semi-irregular cusps with index $3$ on $D_L$ map to regular cusps on $\D_{L_Q}$.
\end{enumerate}

\end{prop}

\section{Discriminant kernel case}
\label{discriminant_kernel_case}
Here, we shall show a structure theorem of the discriminant group when the discriminant kernel may have irregular cusps.
In this section, we assume that the class number of $F$ is 1.
For a rank 1 primitive isotropic sublattice $I$ of $L$ and a generator $e$ of $I$, the quantity $\div(I)$ denotes a generator of the principal ideal $\{\l\ell,e\r\mid\ell\in L\}$.
Note that, unlike the orthogonal case, there is no canonical choice of this quantity.
Let $\Gamma\subset \U(L)(\Z)$ be a finite index subgroup.

In this section and Appendix \ref{app:A}, we assume that $L$ is \textit{even}, that is, $\l\ell,\ell\r\in\Z$ for any $\ell\in L$ in the sense of \cite{Hofmann}.
Note that this implies that the associated quadratic lattice is even.
This corresponds to the assumption in \cite[subsection 4.1]{irregular}.
Let $a,b\in\Z$ be integers with $a\neq 0$ or $b\neq 0$.
This section uses the following notation:
\[\div(I)=
\begin{cases}
\frac{2a+(1+\sqrt{d})b}{2\sqrt{d}} & (d\equiv 1\bmod 4) \\
\frac{a+b\sqrt{d}}{2\sqrt{d}} & (d\equiv 2,3\bmod 4).
\end{cases}\]

\subsection{Preparation}
\begin{lem}[c.f. {\cite[Lemma 4.1]{irregular}}]
\label{discriminant_kernel1}
Assuming $\widetilde{\U}(L)\subset\Gamma$, we have $\sqrt{d}(\overline{I}\otimes I)(\Z)\subset \ZZ(I)_{\Z}$.
Here, 
\[\sqrt{d}(\overline{I}\otimes I)(\Z)\defeq\{\sqrt{d}\lambda (e\otimes e)\mid\lambda\in\Z\}.\]
\end{lem}
\begin{proof}
For $\sqrt{d}\lambda (e\otimes e)\in\sqrt{d}(\overline{I}\otimes I)(\Z)$, we can show that $T_{\sqrt{d}\lambda (e\otimes e)}$ preserves the discriminant group and this gives the inclusion $\sqrt{d}(\overline{I}\otimes I)(\Z)\subset \ZZ(I)_{\Z}$.
\end{proof}

\begin{lem}[c.f. {\cite[Lemma 4.3]{irregular}}]
\label{discriminant_kernel2}
Let  $\Gamma=\widetilde{\U}(L)$.
\begin{enumerate}
    \item For any imaginary quadratic field $F$ with class number $1$, if $I$ is irregular with index $2$, then $2/\div(I)$ is an element of $\OO_F$.
    \item For $F=\Q(\sqrt{-1})$, if $I$ is semi-irregular with index $2$ $\mathrm{(}$resp. irregular index $4\mathrm{)}$, then $(1-\sqrt{-1})/\div(I)$ $\mathrm{(}$resp. $(1+\sqrt{-1})/\div(I)\mathrm{)}$ is an element of $\OO_F$.
    \item For $F=\Q(\sqrt{-3})$, if $I$ is semi-irregular with index $2$ $\mathrm{(}$resp. $\mathrm{(}$semi-$\mathrm{)}$irregular with index $3$, irregular with index $6\mathrm{)}$, then $2/\div(I)$ $\mathrm{(}$resp. $(1-\omega)/\div(I)$,  $(1+\omega)/\div(I)\mathrm{)}$ is an element of $\OO_F$.
\end{enumerate}
\end{lem}
\begin{proof}
(1) Assume $-T_{\sqrt{d}\lambda (e\otimes e)}\in\Gamma=\widetilde{\U}(L)$ for some $\sqrt{d}\lambda (e\otimes e)\in\sqrt{d}(\overline{I}\otimes I)(\Q)$.
Then, for any $v\in I^{\perp}\cap I^{\vee}$, we have
\[-T_{\sqrt{d}\lambda (e\otimes e)}(v)=-v\in v+L\]
because $-T_{\sqrt{d}\lambda (e\otimes e)}$ acts on the discriminant group of $L$ trivially.
This implies that $2v\in L$.
By substituting $v=e'/\div(I)$, we find that $2/\div(I)\in\OO_F$.
We can prove (2) and (3) similarly by calculating $\pm\sqrt{-1}T_{\sqrt{-1}\lambda (e\otimes e)}$ and $\pm\omega T_{\sqrt{-3}\lambda (e\otimes e)}$.

\end{proof}

\begin{lem}[c.f. {\cite[Lemma 4.2]{irregular}}]
\label{discriminant_kernel3}
Let $\widetilde{\U}(L)\subset\Gamma$.
Assume that the following holds for any  $\lambda\in F$; if $2\sqrt{d}\cdot\overline{\div(I)}\lambda$ is an element of $\OO_F$, then $\lambda$ is an element of $\Z$.
Then, $I$ is regular.
\end{lem}
\begin{proof}
For a fixed $\div(I)$, we take an $e'\in L$ such that $\l e,e'\r=\div(I)$.
Now, we shall prove that we can take $e'$ to be an isotropic vector.

For simplicity, we only consider the case of $d\equiv 2,3\bmod 4$.
We assume $\l e',e'\r\neq 0$.
Let $f\defeq (p+q\sqrt{d})e+e'$ for some integers $p,q\in\Z$.
Note that $\l e,f\r=\div(I)$.
Then, since we have
\[\l e,e'\r=\div(I)=\frac{a+b\sqrt{d}}{2\sqrt{d}},\]
it follows that $\l f,f\r=0$ holds if and only if
\begin{align}
\label{eq:even}
    aq+bp=-\l e',e' \r.
\end{align}
Here, $-\l e',e'\r$ is in $\Z$ from the condition that $L$ is even.
On the other hand, by our assumption in lemma, the greatest common divisor of $a$ and $b$ is $1$ so that there exist some integers $p'$ and $q'$ that make the equation $(\ref{eq:even})$ hold.
Hence, it suffices to replace $e'$ with $(p'+q'\sqrt{d})e+e'$.
The same discussion holds for the case of $d\equiv 1\bmod 4$.
Below, we take $e'$ to be an isotropic vector.

First, suppose $I$ is irregular with index 2.
Equivalently, we can assume $-\id\not\in\Gamma$ and $-T_{\sqrt{d}\lambda (e\otimes  e)}\in\Gamma$. 
Since $T_{\sqrt{d}\lambda (e\otimes e)}$ preserves $L$, we have 
\[T_{\sqrt{d}\lambda (e\otimes e)}(e')=e'+2\sqrt{d}\l e',e\r\lambda e\in L.\]
By assumption, $\lambda\in\OO_F$ so that $\sqrt{d}\lambda (e\otimes e)\in \sqrt{d}(\overline{I}\otimes I)(\Z)$.
By Lemma \ref{discriminant_kernel1}, $T_{\sqrt{d}\lambda (e\otimes e)}\in U(I)_{\Z}$, so we obtain  $T_{\sqrt{d}\lambda (e\otimes e)}\in\Gamma$.
This implies $-\id\in\Gamma$, which is a contradiction.

We can give similar proofs for other irregular lattices $I$.
\end{proof}

For analysis of the structures of discriminant groups, we need some invariant decomposition theorem of finitely generated modules over a principal ideal domain.

\begin{prop}
\label{decomp}
Let $\OO$ be a principal ideal domain, $N$ be a finite module over $\OO$ and $p\neq0$ be a prime element in $\OO$.
We assume that an exact sequence
\[0\to\OO/p^m\to N\to \bigoplus_{i=1}^{s}(\OO/p^i)^{\oplus a_i}\to 0\]
exists for some non-negative integers $m,s,a_1,\dots,a_s\in\Z$.
Then, the isomorphism class of $N$ satisfying the above exact sequence corresponds to the pair $(i_0,\dots,i_k,m_0,\dots,m_k)$ such that 
\[\begin{cases}
i_0<\dots<i_k\\
a_{i_{\ell}}>0&(0<\ell\leq k)\\
a_{i_0}>0&(\mathrm{if}\ i_0>0)\\
m_0+\dots+m_k=m&(m_i>0\ \mathrm{for\ any}\ i)\\
0<m_{\ell}<i_{\ell+1}-i_{\ell}&(0\leq\ell<k).
\end{cases}\]
Moreover, 
\[N\cong\begin{cases}
    \OO/p^{m_0}\oplus\displaystyle{\bigoplus_{\ell=1}^k}\Bigl\{(\OO/p^{i_{\ell}})^{\oplus (a_{i_{\ell}}-1)}\oplus\OO/p^{m_{\ell}+i_{\ell}}\Bigr\}\oplus\bigoplus_{\substack{{j\neq i_t} \\\mathrm{for\ any\ }t}}(\OO/p^j)^{\oplus a_j}&(i_0=0)\\
     \displaystyle{\bigoplus_{\ell=0}^k}\Bigl\{(\OO/p^{i_{\ell}})^{\oplus (a_{i_{\ell}}-1)}\oplus\OO/p^{m_{\ell}+i_{\ell}}\Bigr\}\oplus\bigoplus_{\substack{{j\neq i_t}\\ \mathrm{for\ any\ }t}}(\OO/p^j)^{\oplus a_j}&(i_0>0).
\end{cases}\]

\end{prop}

Below, we especially compute the case of class number 1 and discriminant kernels.
In the rest of this section, let $\Gamma=\widetilde{\U}(L)$.
Combining these calculation, it will be possible to narrow down the list of candidates of discriminant groups; see Appendix \ref{app:A} for the classification of $A_L$.

\subsection{Case of $\Q(\sqrt{-1})$}
Let $F=\Q(\sqrt{-1})$.
\begin{prop}
\label{-1_ab_kouho}
\begin{enumerate}
    \item If $I$ is irregular with index $2$, then $\div(I)=\pm1,\pm\sqrt{-1},   \pm1\pm\sqrt{-1}, \pm2, \pm2\sqrt{-1}$.
    \item If $I$ is semi-irregular with index $2$, then $\div(I)=\pm 1\pm\sqrt{-1}$.
    \item If $I$ is irregular with index $4$, then $\div(I)=\pm 1\pm\sqrt{-1}$.
\end{enumerate}
\end{prop}
\begin{proof}
(1) We have
\[\frac{2}{\div(I)}=\frac{4(b+a\sqrt{-1})}{a^2+b^2}.\]
Hence, $2/\div(I)\in\OO_F$ implies $(a,b)=(\pm 1,0), (0,\pm 1), (\pm 1,\pm 1),  (\pm2,0)$,$(0,\pm2)$, $(\pm 2,\pm 2)$,  $(\pm4,0), (0,\pm4)$ and these pairs are the candidates for irregular $I$ with index 2 by Lemma \ref{discriminant_kernel2}.
On the other hand, 
\[2\sqrt{-1}\l e',e\r\lambda=(a+b\sqrt{-1})\lambda.\]
If $(a+b\sqrt{-1})\lambda\in\OO_F$ implies $\lambda\in\Z$, then $I$ is regular by Lemma \ref{discriminant_kernel3}.
In this case, the pairs $(\pm 1,0),  (0,\pm 1), (\pm 1,\pm 1)$ satisfy the condition in Lemma \ref{discriminant_kernel3}; that is, if $(a,b)$ is one of these pairs, then $I$ is regular.
Hence, from the above discussion, if $I$ is irregular, then $(a,b)=(\pm2,0), (0,\pm2),$ $(\pm 2,\pm 2), (\pm4,0),$ $(0,\pm4)$ so that $\div(I)=\pm1,\pm\sqrt{-1},   \pm1\pm\sqrt{-1}, \pm2, \pm2\sqrt{-1}$.

(2) We have
\[\frac{1-\sqrt{-1}}{\div(I)}=\frac{2(a+b)+2(a-b)\sqrt{-1}}{a^2+b^2}.\]
Hence, $(1-\sqrt{-1})/\div(I)\in\OO_F$ implies $(a,b)=(\pm1,0), (0,\pm 1),  (\pm1,\pm1), (\pm2,\pm2)$ and these pairs are the candidates for semi-irregular $I$ with index 2 by Lemma \ref{discriminant_kernel2}.
By performing the same calculation, if $I$ is irregular, then $(a,b)=(\pm 2,\pm 2)$ so that $\div(I)=(\pm 2\pm 2\sqrt{-1})/2=  \pm1\pm\sqrt{-1}$.

We can prove (3) in the same way. 
\end{proof}

\subsection{Case of $\Q(\sqrt{-3})$}
Let $F=\Q(\sqrt{-3})$.
See subsection \ref{anyimagirreg} for the (semi-)index 2 case.
\begin{prop}
\label{-3_ab_kouho}
\begin{enumerate}
    \item     If $I$ is $\mathrm{(}$semi-$\mathrm{)}$irregular with index $3$, then 
    \[\div(I)=\frac{2a+(1+\sqrt{-3})b}{2\sqrt{-3}}\]
    has the candidates listed in Table $\ref{irregular_table1}$.
    \item If $I$ is irregular with index $6$, then 
    \[\div(I)=\frac{2a+(1+\sqrt{-3})b}{2\sqrt{-3}}\]
    has the candidates listed in Table $\ref{irregular_table2}$.
\end{enumerate}

\end{prop}
\begin{proof}
These also follow from a direct calculation.
\end{proof}
\begin{table}[htbp]
    \centering
    \begin{tabular}[t]{|c||c|c|c|c|c|c|c|c|c|c|c|c|c|c|c|c|c|c|c|} \hline
      $a$&$-3$&$-3$&$-2$&$-1$&$-1$&$-1$&$-1$&$0$&$0$&$0$&$0$&$1$&$1$&$1$&$1$&$2$&$3$&$3$\\\hline
      $b$&$0$&$3$&$1$&$-1$&$0$&$1$&$2$&$-3$&$-1$&$1$&$3$&$-2$&$-1$&$0$&$1$&$-1$&$-3$&$0$\\ \hline
    \end{tabular}
        \caption{Candidates for $\div(I)$ for (semi-)irregular $I$ with index 3}
    \label{irregular_table1}
    \end{table}
    
    \begin{table}[hbtp]
    \centering
    \begin{tabular}[t]{|c||c||c|c|c|c|c|c|c|c|c|c|c|c|c|c} \hline
       $a$&$-2$&$-1$&$-1$&$-1$&$-1$&$0$&$0$&$1$&$1$&$1$&$1$&$2$\\ \hline
       $b$&$1$&$-1$&$0$&$1$&$2$&$-1$&$1$&$-2$&$-1$&$0$&$1$&$-1$\\ \hline
    \end{tabular}
    \caption{Candidates for $\div(I)$ for irregular $I$ with index 6}
    \label{irregular_table2}
\end{table}

\subsection{Other cases}
\label{anyimagirreg}
Let $F\neq\Q(\sqrt{-1})$ be an imaginary quadratic field with class number 1, that is, $F=\Q(\sqrt{-2})$, $\Q(\sqrt{-3})$, $\Q(\sqrt{-7})$, $\Q(\sqrt{-11})$, $\Q(\sqrt{-19})$, $\Q(\sqrt{-43})$, $\Q(\sqrt{-67})$, $\Q(\sqrt{-163})$.
Then, by performing a similar calculation to the one above, we can prove the following proposition by using  a computer.
\begin{prop}
\label{d_ab_kouho}
\begin{enumerate}
    \item Let $d\equiv 1\bmod 4$.
If $I$ is irregular with index $2$, then 
\[\div(I)=\frac{2a+(1+\sqrt{d})b}{2\sqrt{d}}\]
has the candidates listed in Table $\ref{irregular_table3}$.
\item Let $d=-2$.
If $I$ is irregular with index $2$, then 
\[\div(I)=\frac{a+b\sqrt{-2}}{2\sqrt{-2}}\]
has the candidates listed in Table $\ref{irregular_table4}$.
\end{enumerate}

\end{prop}
\begin{proof}
These also follow from a direct calculation.
\end{proof}

\begin{table}[h]
\centering
      \begin{tabular}[t]{|c|c||c|c|c|c|c|c|c|c|c|c|c|c|c|c|}
        \cline{1-1}
    $d$ \\ \hline
       \multirow{2}{*}{$-3$}&$a$&$-4$&$-2$&$-2$&$-2$&$-2$&$-2$&$-1$&$-1$&$-1$&$-1$&$0$&$0$&$0$&$0$\\ \cline{2-16}
        &$b$&$2$&$-2$&$0$&$1$&$2$&$4$&$-1$&$0$&$1$&$2$&$-2$&$-1$&$1$&$2$\\ \cline{2-16}
       &$a$&$1$&$1$&$1$&$1$&$2$&$2$&$2$&$2$&$2$&$4$\\ \cline{2-12}
        &$b$&$-2$&$-1$&$0$&$1$&$-4$&$-2$&$-1$&$0$&$2$&$-2$\\ \cline{1-16}
               \multirow{2}{*}{$-7$}&$a$&$-4$&$-3$&$-2$&$-2$&$-1$&$-1$&$-1$&$0$&$0$&$1$&$1$&$1$&$2$&$2$\\ \cline{2-16}
        &$b$&$1$&$-1$&$0$&$4$&$0$&$1$&$2$&$-1$&$1$&$-2$&$-1$&$0$&$-4$&$0$\\ \cline{2-16}
        &$a$&$3$&$4$\\ \cline{2-4}
        &$b$&$1$&$-1$\\ \cline{1-10}
               \multirow{2}{*}{$-11,-19,-43,-67,-163$}&$a$&$-2$&$-2$&$-1$&$-1$&$1$&$1$&$2$&$2$\\ \cline{2-10}
        &$b$&$0$&$4$&$0$&$2$&$-2$&$0$&$-4$&$0$\\ \cline{1-10}
    \end{tabular}
        \caption{Candidates for $\div(I)$ for irregular $I$ with index 2 and $d\equiv 1\bmod 4$}
    \label{irregular_table3}
    \end{table}

\begin{table}[h]

    \begin{tabular}[t]{|c||c|c|c|c|c|c|c|c|c|} \hline
       $a$&$2$&$-2$&$4$&$-4$&$0$&$0$&$0$&$0$\\ \hline
       $b$&$0$&$0$&$0$&$0$&$2$&$-2$&$4$&$-4$\\ \hline
    \end{tabular}
    \vspace{3mm}
    \caption{Candidates for $\div(I)$ for irregular $I$ with index 2 and $d=-2$}
    \label{irregular_table4}
\end{table}

\section{Ramification divisors and canonical singularities}
Now, we consider how irregular cusps affect the geometry of $\overline{\F_L(\Gamma)}$.
The essence of this section is due to \cite[section 7]{irregular}.
\begin{cor}
Let $I$ be a rank $1$ primitive isotropic sublattice of $L$.
Then, $I$ is an irregular with index $m$ if and only if the map $\overline{\mathcal{X}(I)}\to\overline{\mathcal{X}(I)}/\overline{\Gamma(I)_{\Z}}$ ramifies along the unique boundary divisor with ramification index $m$.
Moreover, if we take the quotient $\ZZ(I)^{\star}_{\Z}/\ZZ(I)_{\Z}$, then $\overline{D_L/\ZZ(I)_{\Z}^{\star}}\to\overline{\F_L(\Gamma)}$ does not ramify along the unique boundary divisor.
\end{cor}
\begin{proof}
The first claim follows from Propositions \ref{sp-2-irreg}, \ref{4-irreg}, \ref{sp-2-irreg'}, \ref{sp-3-irreg}, \ref{3-irreg} \ref{6-irreg} and \ref{2-irreg}, and the fact that the unique boundary divisor is $\V(I)_{\C}$.
The second claim follows in the same way as \cite[Proposition 7,2 (2)]{irregular}.
\end{proof}
\begin{rem}
Note that, in the adjoint case, Ma \cite{boundary} proved there is no branch divisor on the boundary of any toroidal compactification of modular varieties.
\end{rem}

Now, let us treat the canonical singularities on the boundary divisors on ball quotients.

\begin{prop}
\label{can_sing_boundary}
If $n\geq 13$ and $d<-3$, then the canonical toroidal compactification $\overline{\F_L(\Gamma)}$ has canonical singularities at the boundary points.
\end{prop}
\begin{proof}
If there is no irregular primitive isotropic sublattice $I\subset L$, then the claim follows from \cite{Behrens}.
Otherwise, in the same way as \cite[Proposition 7.4]{irregular}, we have 
\[\overline{(D_L/\ZZ(I)_{\Z})}/\overline{\Gamma(I)_{\Z}}\cong \overline{(D_L/\ZZ(I)'_{\Z})}/(\Gamma(I)'_{\Z}/\ZZ(I)'_{\Z}).\]
The claim is proved combining this with \cite{Behrens}.
\end{proof}

\section{Low slope cusp form trick}
Let $\L\defeq\OO(-1)\vert_{D_L}$ and $\chi$ be a character of $\Gamma$.
A $\Gamma$-invariant section $\Psi$ of $\L^{\otimes k}\otimes\chi$ is called a modular form of weight $k$ with character $\chi$.
We consider $D_L$ as a Siegel domain of the third kind.
In our setting, for any rank 1 primitive isotropic sublattice $I\subset L$, the corresponding cusp $c_I$ is a point, so we will omit this in the Siegel domain of the third kind and consider $D_L\subset D(I)=\ZZ(I)_{\C}\times \V(I)_{\C}$.
Here, $z$ and $u=(u_1,\dots,u_{n-1})$ denote the local coordinates of $\ZZ(I)_{\C}$ and $\V(I)_{\C}$, respectively.
We take a nowhere vanishing section $s_I$ of $\L$ with respect to $I$ in the same way as in \cite{irregular}.
Then when we write $\Psi=f s_I^{\otimes k}\otimes 1$, the holomorphic function $f$ on $D_L$ satisfies the following modularity condition:
\[f(\gamma[v])=\chi(\gamma)j(\gamma,[v])^{\otimes k}f([v])\quad (\gamma\in\Gamma, [v]\in D_L)\]
where $j(\gamma,[v])$ is the automorphy factor.
We assume $\chi\vert_{\ZZ(I)_{\Z}}=1$ so that $f$ descends to a function on $D_L/\ZZ(I)_{\Z}$.
Then the Fourier expansion of $f$ is 
\[f(z,u)=\sum_{\rho\in \ZZ(I)^{\vee}_{\Z}}\varphi_{\rho}(u)\exp(2\pi\sqrt{-1}\l\rho,z\r).\]

For a generator $w_I$ of $\CC(I)$, we define the vanishing order $v_{I}(\Psi)$ as
\[v_I(\Psi)\defeq\mathrm{min}\{\l\ell,w_I\r\mid\ell\in \ZZ(I)^{\vee}_{\Z},\varphi_{\rho}(\ell)\neq 0\}.\]
Moreover, we define the geometric vanishing order $v_{I,\mathrm{geom}}(\Psi)$ as
\[v_{I,\mathrm{geom}}(\Psi)\defeq
\begin{cases}
v_I(\Psi) & (I:\mathrm{regular})\\
\frac{1}{m}v_I(\Psi) & (I:\mathrm{(semi}\mathchar`-\mathrm{)irregular}\ \mathrm{with}\ \mathrm{index}\ m).
\end{cases}
\]

Then, we can give these vanishing orders a geometrical interpretation.
\begin{prop}[{\cite[Proposition 8.4, 8.5, 8.6]{irregular}}]
\begin{enumerate}
    \item $v_I(\Psi)$ is the vanishing order of $\Psi$ over $\overline{\mathcal{X}(I)}$ along the unique boundary divisor $\V(I)_{\C}$.
    \item If $s_I^{\otimes k}\vert_{\ZZ(I)^{\star}_{\Z}}=1$, then $v_{I,\mathrm{geom}}(\Psi)$ is the vanishing order of $\Psi$ over $\overline{\mathcal{X}(I)'}$ along the unique boundary divisor $V(I)_{\C}'$.
        \item $\L^{\otimes n+1}\otimes\det\cong K_{\overline{\mathcal{X}(I)'}}+\V(I)'_{\C}$ over $\overline{\mathcal{X}(I)'}$.
\end{enumerate}
\end{prop}

The vanishing orders of canonical forms are measured in $\overline{\F_L(\Gamma)}$.
Now, the projection $\overline{\mathcal{X}(I)'}\to\overline{\F_L(\Gamma)}$ does not ramify, so we can measure the order of canonical forms by pulling back to $\overline{\mathcal{X}(I)'}$, i.e., for a modular form $\Psi$ of weight $(n+1)k$ and a corresponding $k$-canonical form $\omega_{\Psi}$, 
\[v_{\V(I)_{\C}}(\omega_{\Psi})=v_{\V(I)'_{\C}}(\pi^{\star}(\omega_{\Psi}))=v_{I, \mathrm{geom}}(\Psi)-k.\]
On the other hand, the projection $\overline{\mathcal{X}(I)}\to\overline{\mathcal{X}(I)'}$ ramifies with index $m$ if $I$ is (semi-)irregular with index $m$ so that 
\[v_{\V(I)_{\C}}(\omega_{\Psi})=\frac{1}{m}v_I(\Psi)-k.\]

\begin{prop} The $k$-canonical form corresponding to a modular form $\Psi$ of weight $(n+1)k$ extends holomorphically over the regular locus of $\overline{\F_L(\Gamma)}$ if and only if the following conditions hold:
    \begin{enumerate}
        \item $v_R(\Psi)\geq (r_i-1)k$ for every irreducible component $R_i$ of the ramification divisors $D_L\to\F_L(\Gamma)$ with ramification index $r_i$. 
        \item $v_I(\Psi)\geq k$ for every regular isotropic sublattice $I\subset L$.
        \item $v_I(\Psi)\geq m_Ik$ for every $\mathrm{(}$semi-$\mathrm{)}$irregular isotropic sublattice $I\subset L$ with index $m_I$.
\end{enumerate}
\end{prop}
\begin{proof}
To conclude the proof, combine the above discussion and \cite[Corollary 8.8]{irregular}.
\end{proof}
\begin{thm}[Low slope cusp form trick]
\label{low_slope_trick_unitary}
Let $F$ be an imaginary quadratic field and $L$ be a Hermitian lattice of signature $(1,n)$ over $\OO_F$.
For a finite index subgroup $\Gamma\subset\U(L)(\Z)$, we assume that there is a non-zero cusp form $\Psi$ of weight $k$ with respect to $\Gamma$ on $D_L$.
In addition, we make the following assumptions.
\begin{enumerate}
    \item $v_R(\Psi)/k>(r_i-1)/(n+1)$ for every irreducible component $R_i$ of the ramification divisors $D_L\to\F_L(\Gamma)$ with ramification index $r_i$.
    \item $v_I(\Psi)/k>1/(n+1)$ for every regular isotropic sublattice $I\subset L$.
    \item $v_I(\Psi)/k>m_I/(n+1)$ for every $\mathrm{(}$semi-$\mathrm{)}$irregular isotropic sublattice $I\subset L$ with index $m_I$.
    \item $n\geq\mathrm{max}_{i,I}\{r_i-2,m_I-1\}$
    \item $\overline{\F_L(\Gamma)}$ has at worst canonical singularities.
\end{enumerate}
Then the ball quotient $\F_L(\Gamma)$ is of general type.
\end{thm}

\begin{rem}
By \cite[Theorem 4]{Behrens}, assumptions (4) and (5) are satisfied if $n\geq 13$ and $d<-3$.
\end{rem}

\begin{proof}
By taking some power of $\Psi$, we may assume that $\Psi$ has trivial character.
Note that $r_i$ is at most 6 by \cite[Corollary 3]{Behrens}.
First, let us assume that $k$ is not divisible by $n+1$.
Let $m\defeq\mathrm{max}_I\{m_I\}\leq 6$ and $r\defeq\mathrm{max}_i\{r_i\}\leq 6$.
By taking some power of $F$, since $n\geq\mathrm{max}\{r-2,m-1\}$, we may assume that
\[\frac{k}{n+1}\geq[\frac{k}{n+1}]+\frac{m-1}{m},\quad  \frac{k}{n+1}\geq[\frac{k}{n+1}]+\frac{r-2}{r-1}.\]
Then, for every ramification divisor with ramification index $r_i$ and every (semi-)irregular isotropic sublattice $I$ with index $m_I$, we have
\[[\frac{m_Ik}{n+1}]=m_I[\frac{k}{n+1}]+1,\quad [\frac{(r_i-1)k}{n+1}]=(r_i-1)[\frac{k}{n+1}]+1.\]
Hence, for $N_0\defeq[\frac{k}{n+1}]+1$, we have
\begin{enumerate}
    \item $v_R(\Psi)\geq (r_i-1)N_0$ for every irreducible component $R_i$ of the ramification divisors $D\to\F_L(\Gamma)$ with ramification index $r_i$.
    \item $v_I(\Psi)\geq N_0$ for every regular isotropic sublattice $I\subset L$.
    \item $v_I(\Psi)\geq m_IN_0$ for every (semi-)irregular isotropic sublattice $I\subset L$ with index $m_I$.
\end{enumerate}

Now we have 
\[V_{\ell}\defeq \Psi^{\ell}M_{((n+1)N_0-k)\ell}(\Gamma)\hookrightarrow M_{(n+1)N_0\ell}(\Gamma).\]
From the above discussion, any element in $V_{\ell}$ holomorphically extends the $\ell N_0$-canonical form over the regular locus of $\overline{\F_L(\Gamma)}$.
On the other hand, Behrens \cite[Theorem 4]{Behrens} showed the canonical singularities of $\overline{\F_L(\Gamma)}$.
Combining this result and Proposition \ref{can_sing_boundary}, we find that $\ell N_0$-canonical forms holomorphically extend over the desingularization of $\overline{\F_L(\Gamma)}$; that is, we can calculate the Kodaira dimension of $\F_L(\Gamma)$ using some desingularization of $\F_L(\Gamma)$.
By Hirzebruch's proportionality principle, the dimension of $V_{\ell}$ grows like   $\ell^{n+1}$ and hence $\F_L(\Gamma)$ is of general type.

Second, we assume that $k$ is divisible by $n+1$.
In this case, we can take $N_0$ in the above discussion to be $k/(n+1)$.
\end{proof}

\begin{rem}
\begin{enumerate}
\item One can construct a non-zero cusp form for $n<13$, which satisfies (1)-(4) in Theorem \ref{low_slope_trick_unitary}, by using a restriction of quasi-pull back of the Borcherds form for $F=\Q(\sqrt{-1}), \Q(\sqrt{-3})$.
    \item  It is known that unitary groups of unimodular Hermitian lattices have no reflections for $F\neq\Q(\sqrt{-1}), \Q(\sqrt{-3})$ \cite{MaedaOdaka, WW}.
Hence, if there exists a cusp form of weight less than $n+1$ which vanishes on irregular cusps with higher order, then $\F_L(\Gamma)$ is of general type in this situation.
\end{enumerate}

\end{rem}

\section{A ball quotient of non-negative Kodaira dimension}
\label{section:non_negative}
To prove that ball quotients are of general type, we need to construct a cusp form of low weight which vanishes on branch divisors with appropriate order by Theorem \ref{low_slope_trick_unitary}.
For the orthogonal modular varieties case, this was done by using Borcherds lift \cite{GHS, Kondo2, irregular}.
For the unitary case, it seems to be difficult to construct a low slope cusp form satisfying Theorem \ref{low_slope_trick_unitary} (5), by using unitary Borcherds lift \cite{Hofmann} because the Borcherds form exists on a 13-dimensional ball.
However, the existence of a cusp form with weaker conditions imposed implies that the Kodaira dimension is non-negative by Freitag's criterion \cite{Freitag}.
In this section, we shall construct a cusp form of canonical weight on a ball quotient and conclude that it has non-negative Kodaira dimension.
Note that in the notation of this paper, the canonical weight is $n+1$.

Let $L_{U\oplus U}$ be an even unimodular  Hermitian lattice of signature $(1,1)$ over $\OO_{\Q(\sqrt{-2})}$ defined by the matrix
\[
\frac{1}{2\sqrt{-2}}
\begin{pmatrix}
0 & 1 \\
-1 & 0 \\
\end{pmatrix}.
\]
Then its associated quadratic lattice $(L_{U\oplus U})_Q$ is $U\oplus U$.

Let $L_{E_8(-1)}$  be an even unimodular  Hermitian lattice of signature $(0,4)$ over $\OO_{\Q(\sqrt{-2})}$ defined by the matrix
\[
-\frac{1}{2}\begin{pmatrix}
2 & 0 & \sqrt{-2}+1&\frac{1}{2}\sqrt{-2}\\
0 & 2 & \frac{1}{2}\sqrt{-2} & 1-\sqrt{-2}\\
1-\sqrt{-2} & -\frac{1}{2}\sqrt{-2} & 2 & 0\\
-\frac{1}{2}\sqrt{-2}& \sqrt{-2}+1 & 0 & 2\\
\end{pmatrix}.
\]
Then its associated quadratic lattice $(L_{E_8(-1)})_Q$ is $E_8(-1)$.

Let $L_{\l-2\r\oplus\l-4\r}$ be an even unimodular Hermitian lattice of signature $(0,1)$ over $\OO_{\Q(\sqrt{-2})}$ defined by the matrix
\[(-1).\]
Then its associated quadratic lattice $(L_{\l-2\r\oplus\l-4\r})_Q$ is $\l-2\r\oplus\l-4\r$.
We define $L_{(\l-2\r\oplus\l-4\r)^{\perp}}$ be the orthogonal complement of $L_{\l-2\r\oplus\l-4\r}$ in $L_{E_8(-1)}$.
Let $L\defeq L_{U\oplus U}\oplus L_{E_8(-1)}\oplus L_{E_8(-1)}\oplus L_{(\l-2\r\oplus\l-4\r)^{\perp}}$ be a Hermitian lattice of signature $(1,12)$ over $\OO_{\Q(\sqrt{-2})}$ whose associated quadratic lattice is $U\oplus U\oplus E_8(-1)\oplus E_8(-1)\oplus(\l -2\r\oplus\l-4\r)^{\perp}$.

For $II_{2,26}\defeq U\oplus U\oplus E_8(-1)\oplus E_8(-1)\oplus E_8(-1)$, we embed $L_Q\hookrightarrow II_{2,26}$ by Nikulin's theorem.
On the Hermitian symmetric domain $\D_{II_{2,26}}$, there exists the Borcherds form $\Phi_{12}$, a modular form of weight 12 with respect to $\O^+(II_{2,26})$ with character $\det$.
This is obtained by using the Borcherds lift of the inverse of Ramanujan's tau function.

\begin{prop}
There exists a non-zero cusp form $\Psi_{13}$ of weight $13$ with respect to $\widetilde{\U}(L)$ with character $\det$.
\end{prop}
\begin{proof}
Since the complement of $L_Q$ in $II_{2,26}$ has exactly two  $(-2)$-vectors, by \cite[Theorem 8.2]{GHS4}, the quasi-pull back $f_{13}$ of $\Phi_{12}$ is a cusp form of weight 12+2/2=13 with respect to $\widetilde{\O}^+(L_Q)$ with character $\det$.
Then by restricting $f_{13}$ to $D_L$, we obtain a cusp form $\Psi_{13}\defeq\iota^{\star}f_{13}$ of weight 13 with respect to $\widetilde{\U}(L)$ with character $\det$ on a 12-dimensional ball $D_L$.
\end{proof}
Therefore, since the canonical bundle on $D_L$ is isomorphic to $\OO(-13)$, by Freitag's criterion \cite{Freitag}, 
 we have the following.
\begin{prop}
The ball quotient $\F_L(\widetilde{\U}(L))$ has non-negative Kodaira dimension.
\end{prop}

\section{Examples}
\label{section:examples}
In this section, we give, as examples, the irregular cusps with any branch indices for any imaginary quadratic fields with class number 1.

\subsection{Case of $\Q(\sqrt{-1})$}
Let $\eta\defeq 1+\sqrt{-1}$.
\begin{ex}
Let $a=2b+1$ be an integer with $b\geq0$ and $L$ be a Hermitian lattice of signature $(1,b+1)$ defined by 

\[\l-1\r^{\oplus b}\oplus
\begin{pmatrix}
0 & \eta^a \\
\overline{\eta}^a & 0 \\
\end{pmatrix}.
\]
Then, we have 
\[A_L\cong (\OO_{\Q(\sqrt{-1})}/\eta^2)^{\oplus b}\oplus(\OO_{\Q(\sqrt{-1})}/\eta^{a+2})^{\oplus 2}.\]
We put 
\[M\defeq\begin{pmatrix}
0 & \eta^a \\
\overline{\eta}^a & 0 \\
\end{pmatrix}.
\]
We take a generator $e_1,\dots,e_b$ of $\l-1\r^b$ and 
$v,w$ of $M$.
In other words, $\l e_i,e_j\r=-\delta_{ij}$ and $\l v,v\r=\l w,w\r=0$, $\l v,w\r=\eta^a$.
We define $A_v$ to be the subgroup of $A_{M}$ generated by $v/\eta^{a+2}$.

Now we take an isotropic vector
\[\ell\defeq e_1+\dots+e_b+v+w.\]

Let 
\[\Gamma\defeq\widetilde{\U}(L)^v\defeq\{g\in\U(L)(\Z)\mid g\vert_{A_v}=\id\}\] .
Then, we have
\[\begin{cases}
-\id\in\Gamma,\sqrt{-1}\id\not\in\Gamma&(a=-1)\\
-\id,\sqrt{-1}\id\not\in\Gamma&(a\geq 0).\\
\end{cases}\]

Now for $\lambda\defeq 1/2^{b+1}$, we can show 
\[-\sqrt{-1}T_{\lambda\sqrt{-1}(\ell\otimes\ell)}\in\Gamma\]
by our assumption on $a$ and $b$, that is, 
\[-\sqrt{-1}T_{\lambda\sqrt{-1}(\ell\otimes\ell)}(\frac{v}{\eta^{a+2}})=\frac{v}{\eta^{a+2}}\in A_v,\]
\[-\sqrt{-1}T_{\lambda\sqrt{-1}(\ell\otimes\ell)}(e_i)\in L,\  -\sqrt{-1}T_{\lambda\sqrt{-1}(\ell\otimes\ell)}(w)\in L.\]
Hence, $\ell$ defines an irregular sublattice of $L$ with index 4.
\end{ex}

\begin{ex}
Let  $L$ be a Hermitian lattice of signature $(1,3)$ defined by 

\[\l-\frac{1}{2}\r^{\oplus 2}\oplus
\begin{pmatrix}
0 & \frac{\overline{\eta}}{2} \\
\frac{\eta}{2} & 0 \\
\end{pmatrix}.
\]
Then we have 
\[A_L\cong (\OO_{\Q(\sqrt{-1})}/\eta)^{\oplus 2}.\]

We put 
\[M_1\defeq \l-\frac{1}{2}\r^{\oplus 2},\ M_2\defeq \begin{pmatrix}
0 & \frac{\overline{\eta}}{2} \\
\frac{\eta}{2} & 0 \\
\end{pmatrix}.\]

We take a generator $e,f$ of $M_1$ and $v, w$ of $M_2$.
We define $A_v$ to be the subgroup of $A_{L}$ generated by $v/\eta$.

Now we take an isotropic vector
\[\ell\defeq e+f+v+w.\]

Let 
\[\Gamma\defeq\widetilde{\U}(L)^v\defeq\{g\in\U(L)(\Z)\mid g\vert_{A_v}=\id\}.\]
We put $\lambda\defeq -1$.
Then, we have 
\[-\id\in\Gamma,\ \sqrt{-1}\id\not\in\Gamma,\ \sqrt{-1}T_{-\sqrt{-1}(\ell\otimes\ell)}\in\Gamma.\]
Hence, $\ell$ defines an semi-irregular sublattice of $L$ with index 2.
\end{ex}

\subsection{Case of $\Q(\sqrt{-3})$}
Let $\omega\defeq(-1+\sqrt{-3})/2$.
\begin{ex}
Let $L$ be a Hermitian lattice of signature $(1,2)$ defined by 

\[\l-1\r\oplus
\begin{pmatrix}
0 & \omega \\
\overline{\omega} & 0 \\
\end{pmatrix}.
\]
Then we have 
\[A_L\cong (\OO_{\Q(\sqrt{-3})}/\sqrt{-3})^{\oplus 3}.\]

We take a generator $e, v,w$ of $L$ with $\l e,e\r=-1$, $\l v,v\r=\l w,w\r=0$ and $\l v,w\r=\omega$.
We define $A_w$ to be the subgroup of $A_L$ generated by $w/\sqrt{-3}$.

Now we take an isotropic vector
\[\ell\defeq e+v+w.\]
Let 
\[\Gamma\defeq\widetilde{\U}(L)^w\defeq\{g\in\U(L)(\Z)\mid g\vert_{A_w}=\id\}.\]
Then, we have
\[\omega\id\not\in\Gamma.\]

Now for $\lambda\defeq -1/2$, we can show 
\[\omega T_{\lambda\sqrt{-3}(\ell\otimes\ell)}\in\Gamma,\ -\omega T_{\lambda\sqrt{-3}(\ell\otimes\ell)}\not\in\Gamma.\]
Hence, $\ell$ defines an irregular sublattice of $L$ with index 3.
\end{ex}

\begin{ex}
Let $L$ be a Hermitian lattice of signature $(1,4)$ defined by 

\[\l-1\r^{\oplus 3}\oplus
\begin{pmatrix}
0 & \frac{3+\sqrt{-3}}{2} \\
\frac{3-\sqrt{-3}}{2} & 0 \\
\end{pmatrix}.
\]
We have 
\[A_L\cong (\OO_{\Q(\sqrt{-3})}/\sqrt{-3})^{\oplus 3}\oplus (\OO_{\Q(\sqrt{-3})}/3)^{\oplus 2}.\]

We take a generator $e_1,e_2,e_3,v,w$ of $L$ with $\l e_i,e_j\r=-\delta_{ij}$, $\l v,v\r=\l w,w\r=0$ and $\l v,w\r=(3+\sqrt{-3})/2$.
We define $A_v$ to be the subgroup of $A_L$ generated by $v/3$.

Now we take an isotropic vector
\[\ell\defeq e_1+e_2+e_3+f+v+w.\]
Let 
\[\Gamma\defeq\widetilde{\U}(L)^v\defeq\{g\in\U(L)(\Z)\mid g\vert_{A_v}=\id\}.\]
Then, we have
\[-\id,\omega\id\not\in\Gamma.\]

Now for $\lambda\defeq -1/6$, we can show 
\[-\omega T_{\lambda\sqrt{-3}(\ell\otimes\ell)}\in\Gamma.\]
Hence, $\ell$ defines an irregular sublattice of $L$ with index 6.
\end{ex}

\subsection{General case}
In this subsection, let $F=\Q(\sqrt{d})$ be an imaginary quadratic field with $d\neq -1$ and  $\eta\defeq\sqrt{d}$.
\begin{ex}
Let $L$ be a Hermitian lattice of signature $(1,1)$ defined by 

\[
\begin{pmatrix}
0 & \eta \\
\overline{\eta} & 0 \\
\end{pmatrix}.
\]

We take a generator $v,w$ of $L$.
We define $A_v$ to be the subgroup of $A_{L}$ generated by

\[\begin{cases}
\frac{v}{2\eta^2}& (d\equiv 2,3\bmod 4) \\
\frac{v}{\eta^2}& (d\equiv 1\bmod 4) .
\end{cases}\]

Now we take an isotropic vector
\[\ell\defeq e+f+v+w.\]
Let 
\[\Gamma\defeq\widetilde{\U}(L)^v\defeq\{g\in\U(L)(\Z)\mid g\vert_{A_v}=\id\}.\]
Then, we have
\[-\id\not\in\Gamma\]
if $d\neq -1$.

Now for $\lambda\defeq -1/d$, we can show 
\[-T_{\lambda\sqrt{d}(\ell\otimes\ell)}\in\Gamma.\]
Hence, $\ell$ defines an irregular sublattice of $L$ with index 2.
\end{ex}

\appendix
\section{Classification of discriminant groups}
\label{app:A}
Below, for simplicity, we use the following concise notation for $\OO_F$-modules.
For $\eta_1,\eta_2\in\OO_F$ and $a,b,c,d\in\Z_{\geq 0}$, we write
\[a\cdot\eta^b\oplus c\cdot\eta^d\]
to denote the $\OO_F$-module
\[(\OO_F/\eta^b)^{\oplus a}\oplus (\OO_F/\eta^d)^{\oplus c}.\]

Here, we give the candidates for discriminant groups when the discriminant kernel may have irregular cusps, over any imaginary quadratic fields with class number 1.
We use the notations and assumptions in Section \ref{discriminant_kernel_case}.
Below, for each quantity $\div(I)$, we list possible candidates for $A_L$.
\subsection{Case of $\Q(\sqrt{-1})$}
\label{app:-1}
Let $\eta\defeq 1+\sqrt{-1}$ and $a,b$ be non-negative integers.
\subsubsection{$\mathbf{Index\ 2\ case}$}
Let $I$ be an irregular isotropic sublattice of $L$ with index 2 with respect to $\widetilde{\U}(L)$.
Then, by Proposition \ref{-1_ab_kouho}, we have $\div(I)\equiv 1$, $1+\sqrt{-1}$ or $2$ modulo  $\OO_{\Q(\sqrt{-1})}^{\times}$.
\newline

If $\div(I)\equiv 1$, the candidates are 
\[
a\cdot\eta\oplus b\cdot\eta^2\oplus\eta^{c},\quad 
a\cdot\eta\oplus b\cdot\eta^2\oplus\eta^{d_1}\oplus\eta^{d_2}\]
where $c=3,4,5,6$, $(d_1,d_2)=(3,3), (3,4), (3,5)$.
\newline

If $\div(I)\equiv 1+\sqrt{-1}$, the candidates are 
\[
a\cdot\eta\oplus b\cdot\eta^2\oplus\eta^{c},\quad
a\cdot\eta\oplus b\cdot\eta^2\oplus\eta^{d_1}\oplus\eta^{d_2}\]
where $c=4,5,6,7,8$, $(d_1,d_2)=(1,7), (3,3), (3,4), (3,5), (3,6), (3,7), (4,4),$ $(4,5),$ $(4,6), (5,5)$.
\newline

If $\div(I)\equiv 2$, the candidates are 
\[
a\cdot\eta\oplus b\cdot\eta^2\oplus\eta^{c},\quad  
a\cdot\eta\oplus b\cdot\eta^2\oplus\eta^{d_1}\oplus\eta^{d_2}
\]
where $c=6,7,8,9,10$, $(d_1,d_2)=(3,5), (3,6), (3,7), (3,8), (4,4), (4,5), (4,6)$, $(4,7),$ $(4,8), (5,5)$, $(5,6),$ $(5,7),$  $(6,6)$.

\subsubsection{$\mathbf{Semi\mathchar`-irregular\ with\ index\ 2\ or\ index\ 4\ case}$}
Let $I$ be a semi-irregular isotropic sublattice of $L$ with index 2 or irregular with index 4 with respect to $\widetilde{\U}(L)$.
Then, by Proposition \ref{-1_ab_kouho}, we have $\div(I)\equiv 1+\sqrt{-1}$ modulo  $\OO_{\Q(\sqrt{-1})}^{\times}$.
\newline

If $\div(I)\equiv 1+\sqrt{-1}$, the candidates are 
\[
a\cdot\eta\oplus\eta^{c},\quad 
a\cdot\eta\oplus\eta^{d_1}\oplus\eta^{d_2}\]
where $c=5,6,7$, $(d_1,d_2)=(1,6), (2,4), (2,5), (2,6), (3,3), (3,4), (3,5), (4,4)$.

\subsection{Case of $\Q(\sqrt{-2})$}
\label{app:-2}
Let $\eta\defeq\sqrt{-2}$ and $a,b$ be non-negative integers.
Let $I$ be an irregular isotropic sublattice of $L$ with index 2 with respect to $\widetilde{\U}(L)$.
Then, by Proposition \ref{d_ab_kouho} (2), we have $\div(I)\equiv 1/\sqrt{-2}$, $1$, $\sqrt{-2}$ or $2$ modulo  $\OO_{\Q(\sqrt{-2})}^{\times}$.
\newline

If $\div(I)\equiv 1/\sqrt{-2}$, the candidates are 
\[
a\cdot\eta\oplus b\cdot\eta^2\oplus\eta^{c},\quad 
a\cdot\eta\oplus b\cdot\eta^2\oplus\eta^{d_1}\oplus\eta^{d_2}\]
where $c=3,4,5,6$, $(d_1,d_2)=(3,3), (3,4), (3,5)$.
\newline

If $\div(I)\equiv 1$, the candidates are 
\[
a\cdot\eta\oplus b\cdot\eta^2\oplus\eta^{c},\quad
a\cdot\eta\oplus b\cdot\eta^2\oplus\eta^{d_1}\oplus\eta^{d_2}\]
where $c=4,5,6,7,8$, $(d_1,d_2)=(2,7), (3,3), (3,4), (3,5), (3,6), (3,7), (4,4), (4,5), (4,6)$, $(5,5)$.
\newline

If $\div(I)\equiv \sqrt{-2}$, the candidates are 
\[
a\cdot\eta\oplus b\cdot\eta^2\oplus\eta^{c},\quad 
a\cdot\eta\oplus b\cdot\eta^2\oplus\eta^{d_1}\oplus\eta^{d_2}\]
where $c=6,7,8,9,10$, $(d_1,d_2)=(3,5), (3,6), (3,7), (3,8), (4,4), (4,5), (4,6)$, $(4,7)$, $(4,8)$, $(5,5),$  $(5,6), (5,7),  (6,6)$.
\newline

If $\div(I)\equiv 2$, the candidates are 
\[
a\cdot\eta\oplus b\cdot\eta^2\oplus\eta^{c},\quad
a\cdot\eta\oplus b\cdot\eta^2\oplus\eta^{d_1}\oplus\eta^{d_2}\]
where $c=6,7,8,9,10,11,12$, $(d_1,d_2)=(1,11), (3,7), (3,8), (3,9), (3,10)$, $(3,11),$ $(4,6),$ $(4,7),$ $(4,8),$ $(4,9),$ $(4,10),$ $(5,5),$ $(5,6), (5,7), (5,8), (5,9), (6,6), (6,7), (6,8), (7,7)$.

\subsection{Case of $\Q(\sqrt{-3})$}
\label{app:-3}
Let $\eta\defeq\sqrt{-3}$, $\delta\defeq2$ and $a,b$ be  non-negative integers.
\subsubsection{$\mathbf{Index\ 2\ case}$}
Let $I$ be an irregular isotropic sublattice of $L$ with index 2 with respect to $\widetilde{\U}(L)$.
Then, by Proposition \ref{d_ab_kouho} (1), we have $\div(I)\equiv 1/\sqrt{-3}$, $1$, $2/\sqrt{-3}$ or $2$ modulo  $\OO_{\Q(\sqrt{-3})}^{\times}$.
\newline

If $\div(I)\equiv 1/\sqrt{-3}$, then $A_L$ is isomorphic to
$a\cdot\delta$
as $\OO_{\Q(\sqrt{-3})}$-modules.
\newline 

If $\div(I)\equiv 1$, the candidates are 
\[
a\cdot \delta\oplus\eta^2,\quad 
a\cdot\delta\oplus 2\cdot\eta.\]
\newline

If $\div(I)\equiv 2/\sqrt{-3}$, the candidates are 
\[
a\cdot\delta\oplus \delta^{c}, \quad
a\cdot\delta\oplus 2\cdot\delta^{2}\]
where $c=0,2,3$.
\newline

If $\div(I)\equiv2$, the candidates are 
\[a\cdot\delta\oplus\eta^{2},\quad
a\cdot\delta\oplus 2\cdot\eta^{2},\quad 
a\cdot\delta\oplus\delta^{c}\oplus\eta^{2},\quad
a\cdot\delta\oplus\delta^{c}\oplus 2\cdot\eta,\quad  
a\cdot\delta\oplus 2\cdot\delta^2\oplus\eta^2,\quad 
a\cdot\delta\oplus 2\cdot\delta^2\oplus 2\cdot\eta\]
where $c=2, 3$.

\subsubsection{$\mathbf{Index\ 3\ case}$}
Let $I$ be an irregular isotropic sublattice of $L$ with index 3 with respect to $\widetilde{\U}(L)$.
Then, by Proposition \ref{-3_ab_kouho} (1), we have $\div(I)\equiv 1/\sqrt{-3}$, $1$, $\sqrt{-3}$ modulo  $\OO_{\Q(\sqrt{-3})}^{\times}$.
\newline

If $\div(I)\equiv 1/\sqrt{-3}$, then $A_L$ is isomorphic to
$a\cdot \delta$
as $\OO_{\Q(\sqrt{-3})}$-modules.
\newline 

If $\div(I)\equiv 1$, the candidates are 
\[
a\cdot\eta\oplus b\cdot\eta^2\oplus\eta^{c},\quad
a\cdot\eta\oplus b\cdot\eta^2\oplus 2\cdot\eta^3\]
where $c=0, 3, 4$.
\newline

If $\div(I)\equiv\sqrt{-3}$, the candidates are 
\[a\cdot\eta\oplus b\cdot\eta^2\oplus\eta^{c},\quad 
a\cdot\eta\oplus b\cdot\eta^2\oplus 2\cdot\eta^3,\quad
a\cdot\eta\oplus b\cdot\eta^2\oplus\eta^3\oplus\eta^d\]
where $c=0, 3, 4, 5, 6$, $d=4,5$.

\subsubsection{$\mathbf{Index\ 6\ case}$}
Let $I$ be an irregular isotropic sublattice of $L$ with index 6 with respect to $\widetilde{\U}(L)$.
Then, by Proposition \ref{-3_ab_kouho} (2), we have $\div(I)\equiv 1/\sqrt{-3}$, $1$ modulo  $\OO_{\Q(\sqrt{-3})}^{\times}$.
\newline

If $\div(I)\equiv 1/\sqrt{-3}$, then $A_L$ is trivial, that is, $L$ is unimodular lattice.
\newline 

If $\div(I)\equiv\sqrt{-3}$, the candidates are 
\[
\eta^2,\quad
2\cdot\eta.\]

\subsection{Case of $\Q(\sqrt{-7})$}
\label{app:-7}
Let $\eta_1\defeq(1+\sqrt{-7})/2$, $\eta_2\defeq(-1+\sqrt{-7})/2$, 
$\delta\defeq\sqrt{-7}$ and $a,b$ be non-negative integers.
Let $I$ be an irregular isotropic sublattice of $L$ with index 2 with respect to $\widetilde{\U}(L)$.
Then, by Proposition \ref{d_ab_kouho} (1), we have $\div(I)\equiv 1/\sqrt{-7}$, $1$, $\eta_1/\sqrt{-7}$, $\eta_2/\sqrt{-7}$, $\eta_1\eta_2/\sqrt{-7}$, $\eta_1$, $\eta_2$ or $\eta_1\eta_2$ modulo $\OO_{\Q(\sqrt{-7})}^{\times}$.
\newline

If $\div(I)\equiv 1/\sqrt{-7}$, then $A_L$ is isomorphic to 
$a\cdot \eta_1\oplus b\cdot\eta_2$
 as  $\OO_{\Q(\sqrt{-7})}$-modules.
\newline

If $\div(I)\equiv 1$, the candidates are 
\[
a\cdot\eta_1\oplus b\cdot\eta_2\oplus \delta^2,\quad 
a\cdot\eta_1\oplus b\cdot\eta_2\oplus 2\cdot\delta.\]
\newline

If $\div(I)\equiv \eta_1/\sqrt{-7}$, the candidates are 
\[
(a-2)\cdot\eta_1\oplus a\cdot\eta_2\oplus 2\cdot\eta_1^2,\quad
(a-1)\cdot\eta_1\oplus a\cdot\eta_2\oplus\eta_1^3,\quad
a\cdot\eta_1\oplus a\cdot\eta_2\oplus \eta_1^2,\quad
(a+2)\cdot\eta_1\oplus a\cdot\eta_2.
\]
\newline

If $\div(I)\equiv \eta_2/\sqrt{-7}$, the candidates are 
\[
(a-2)\cdot\eta_2\oplus a\cdot\eta_1\oplus 2\cdot\eta_2^2,\quad
(a-1)\cdot\eta_2\oplus a\cdot\eta_1\oplus\eta_2^3,\quad
a\cdot\eta_2\oplus a\cdot\eta_1\oplus \eta_2^2,\quad
(a+2)\cdot\eta_2\oplus a\cdot\eta_1.
\]
\newline

If $\div(I)\equiv\eta_1$, the candidates are 
\[
(a-2)\cdot \eta_1\oplus a\cdot\eta_2\oplus 2\cdot\eta_1^2\oplus \delta^2,\quad
(a-2)\cdot \eta_1\oplus a\cdot\eta_2\oplus  2\cdot\eta_1^2\oplus 2\cdot\delta,\quad
(a-1)\cdot \eta_1\oplus a\cdot\eta_2\oplus\eta_1^3\oplus\delta^2,\]
\[(a-1)\cdot \eta_1\oplus a\cdot\eta_2\oplus\eta_1^3\oplus 2\cdot\delta,\quad
a\cdot \eta_1\oplus a\cdot\eta_2\oplus\eta_1^2\oplus 2\cdot\delta,\quad
a\cdot \eta_1\oplus a\cdot\eta_2\oplus\eta_1^2\oplus\delta^2,\]
\[(a+2)\cdot \eta_1\oplus a\cdot\eta_2\oplus\delta^2,\quad
(a+2)\cdot \eta_1\oplus a\cdot\eta_2\oplus 2\cdot\delta.\]
\newline

If $\div(I)\equiv\eta_2$, the candidates are 
\[
(a-2)\cdot \eta_2\oplus a\cdot\eta_1\oplus 2\cdot\eta_2^2\oplus \delta^2,\quad
(a-2)\cdot \eta_2\oplus a\cdot\eta_1\oplus  2\cdot\eta_2^2\oplus 2\cdot\delta,\quad
(a-1)\cdot \eta_2\oplus a\cdot\eta_1\oplus\eta_2^3\oplus\delta^2,\]
\[(a-1)\cdot \eta_2\oplus a\cdot\eta_1\oplus\eta_2^3\oplus 2\cdot\delta,\quad
a\cdot \eta_2\oplus a\cdot\eta_1\oplus\eta_2^2\oplus 2\cdot\delta,\quad
a\cdot \eta_2\oplus a\cdot\eta_1\oplus\eta_2^2\oplus\delta^2,\]
\[(a+2)\cdot \eta_2\oplus a\cdot\eta_1\oplus\delta^2,\quad
(a+2)\cdot \eta_2\oplus a\cdot\eta_1\oplus 2\cdot\delta.\]
\newline

If $\div(I)\equiv\eta_1\eta_2/\sqrt{-7}$, the candidates are 
\[
(a-2)\cdot\eta_1\oplus (a-2) \cdot\eta_2\oplus 2\cdot\eta_1^2\oplus 2\cdot \eta_2^2,\quad
(a-2)\cdot\eta_1\oplus (a-1) \cdot\eta_2\oplus 2\cdot\eta_1^2\oplus\eta_2^3,\quad
(a-2)\cdot\eta_1\oplus a \cdot\eta_2\oplus \eta_1^2\oplus 2\cdot \eta_2^2,\]
\[(a-2)\cdot\eta_1\oplus a \cdot\eta_2\oplus 2\cdot\eta_1^2\oplus \eta_2^2,\quad
(a-2)\cdot\eta_1\oplus (a+1) \oplus \eta_1^2\oplus\eta_2^3,\quad
(a-2)\cdot\eta_1\oplus (a+2) \oplus \eta_1^2\oplus \eta_2^2,\]
\[(a-1)\cdot\eta_1\oplus (a-2) \cdot\eta_2\oplus\eta_1^3\oplus 2\cdot \eta_2^2,\quad
(a-1)\cdot\eta_1\oplus (a-1)\cdot\eta_2\oplus\eta_1^3\oplus\eta_2^3,\quad
(a-1)\cdot\eta_1\oplus a \cdot\eta_2\oplus\eta_1^3\oplus \eta_2^2,\]
\[a\cdot\eta_1\oplus (a-2) \cdot\eta_2\oplus \eta_1^2\oplus 2\cdot \eta_2^2,\quad
a\cdot\eta_1\oplus (a-2) \cdot\eta_2\oplus 2\cdot\eta_1^2\oplus \eta_2^2,\quad
a\cdot\eta_1\oplus (a-1)\cdot\eta_2\oplus \eta_1^2\oplus\eta_2^3,\]
\[a\cdot\eta_1\oplus a \cdot\eta_2\oplus 2\cdot\eta_1^2,\quad
a\cdot\eta_1\oplus a \cdot\eta_2\oplus 2\cdot \eta_2^2,\quad
a\cdot\eta_1\oplus a \cdot\eta_2\oplus \eta_1^2\oplus \eta_2^2,\]
\[a\cdot\eta_1\oplus (a+1) \cdot\eta_2\oplus\eta_2^3,\quad
a\cdot\eta_1\oplus (a+2) \cdot\eta_2\oplus \eta_1^2,\quad
a\cdot\eta_1\oplus (a+2) \cdot\eta_2\oplus \eta_2^2,\]
\[(a+1)\cdot\eta_1\oplus (a-2) \cdot\eta_2\oplus\eta_1^3\oplus \eta_2^2,\quad
(a+1)\cdot\eta_1\oplus a \cdot\eta_2\oplus\eta_1^3,\quad
(a+2)\cdot\eta_1\oplus (a-2) \cdot\eta_2\oplus \eta_1^2\oplus \eta_2^2,\]
\[(a+2)\cdot\eta_1\oplus a \cdot\eta_2\oplus \eta_1^2,\quad
(a+2)\cdot\eta_1\oplus a \cdot\eta_2\oplus \eta_2^2,\quad
(a+2)\cdot\eta_1\oplus (a+2) \cdot\eta_2.
\]
\newline

If $\div(I)\equiv\eta_1\eta_2$, the candidates are 
\[
(a-2)\cdot\eta_1\oplus (a-2)\cdot \eta_2\oplus 2\cdot\eta_1^2\oplus 2\cdot \eta_2^2\oplus\delta^2,\quad
(a-2)\cdot\eta_1\oplus (a-2)\cdot \eta_2\oplus 2\cdot\eta_1^2\oplus 2\cdot \eta_2^2\oplus 2\cdot\delta,\quad
(a-2)\cdot\eta_1\oplus (a-1)\cdot\eta_2\oplus 2\cdot\eta_1^2\oplus\eta_2^3\oplus\delta^2,\]
\[(a-2)\cdot\eta_1\oplus (a-1)\cdot\eta_2\oplus 2\cdot\eta_1^2\oplus\eta_2^3\oplus 2\cdot\delta,\quad
(a-2)\cdot\eta_1\oplus a\cdot \eta_2\oplus 2\cdot\eta_1^2\oplus\eta_2^2\oplus\delta^2,\quad
(a-2)\cdot\eta_1\oplus a\cdot \eta_2\oplus 2\cdot\eta_1^2\oplus\eta_2^2\oplus 2\cdot\delta,\]
\[(a-2)\cdot\eta_1\oplus (a+2)\cdot \eta_2\oplus 2\cdot\eta_1^2\oplus\delta^2,\quad
(a-2)\cdot\eta_1\oplus (a+2)\cdot \eta_2\oplus 2\cdot\eta_1^2\oplus 2\cdot\delta,\quad
(a-1)\cdot \eta_1\oplus (a-2)\cdot \eta_2\oplus\eta_1^3\oplus 2 \cdot\eta_2^2\oplus\delta^2,\]
\[(a-1)\cdot \eta_1\oplus (a-2)\cdot \eta_2\oplus\eta_1^3\oplus 2 \cdot\eta_2^2\oplus 2\cdot\delta,\quad
(a-1)\cdot \eta_1\oplus (a-1)\cdot\eta_2\oplus\eta_1^3\oplus\eta_2^3\oplus\delta^2,\quad
(a-1)\cdot \eta_1\oplus (a-1)\cdot\eta_2\oplus\eta_1^3\oplus\eta_2^3\oplus 2\cdot\delta\]
\[(a-1)\cdot \eta_1\oplus a\cdot \eta_2\oplus\eta_1^3\oplus\eta_2^2\oplus\delta^2,\quad
(a-1)\cdot \eta_1\oplus a\cdot \eta_2\oplus\eta_1^3\oplus\eta_2^2\oplus 2\cdot\delta,\quad
(a-1)\cdot \eta_1\oplus (a+2)\cdot \eta_2\oplus\eta_1^3\oplus\delta^2\]
\[(a-1)\cdot \eta_1\oplus (a+2)\cdot \eta_2\oplus\eta_1^3\oplus 2\cdot\delta,\quad
a\cdot\eta_1\oplus (a-2)\cdot \eta_2\oplus\eta_1^2\oplus 2 \cdot\eta_2^2\oplus\delta^2,\quad
a\cdot\eta_1\oplus (a-2)\cdot \eta_2\oplus\eta_1^2\oplus 2 \cdot\eta_2^2\oplus 2\cdot\delta,\]
\[a\cdot\eta_1\oplus (a-1)\cdot\eta_2\oplus\eta_1^2\oplus\eta_2^3\oplus\delta^2,\quad
a\cdot\eta_1\oplus (a-1)\cdot\eta_2\oplus\eta_1^2\oplus\eta_2^3\oplus 2\cdot\delta,\quad
a\cdot\eta_1\oplus a\cdot \eta_2\oplus\eta_1^2\oplus\eta_2^2\oplus\delta^2,\]
\[a\cdot\eta_1\oplus a\cdot \eta_2\oplus\eta_1^2\oplus\eta_2^2\oplus 2\cdot\delta,\quad
a\cdot\eta_1\oplus (a+2)\cdot \eta_2\oplus\eta_1^2\oplus\delta^2,\quad
a\cdot\eta_1\oplus (a+2)\cdot \eta_2\oplus\eta_1^2\oplus 2\cdot\delta,\]
\[(a+2)\cdot\eta_1\oplus (a-2)\cdot \eta_2\oplus 2 \cdot\eta_2^2\oplus\delta^2,\quad
(a+2)\cdot\eta_1\oplus (a-2)\cdot \eta_2\oplus 2 \cdot\eta_2^2\oplus 2\cdot\delta,\quad
(a+2)\cdot\eta_1\oplus (a-1)\cdot\eta_2\oplus\eta_2^3\oplus\delta^2,\]
\[(a+2)\cdot\eta_1\oplus (a-1)\cdot\eta_2\oplus\eta_2^3\oplus 2\cdot\delta,\quad
(a+2)\cdot\eta_1\oplus a\cdot \eta_2\oplus\eta_2^2\oplus\delta^2,\quad
(a+2)\cdot\eta_1\oplus a\cdot \eta_2\oplus\eta_2^2\oplus 2\cdot\delta,\]
\[(a+2)\cdot\eta_1\oplus (a+2)\cdot \eta_2\oplus\delta^2,\quad
(a+2)\cdot\eta_1\oplus (a+2)\cdot \eta_2\oplus 2\cdot\delta.\]
\newline

\subsection{Other cases}
\label{app:d}
Let $F=\Q(\sqrt{d})$, where $d=-11$, $-19$, $-43$, $-67$ or $-163$, $\eta\defeq\sqrt{d}$,  
$\delta\defeq2$ and $a,b$ be non-negative integers.
Let $I$ be an irregular isotropic sublattice of $L$ with index 2 with respect to $\widetilde{\U}(L)$.
Then, by Proposition \ref{d_ab_kouho} (1), we have $\div(I)\equiv 1/\sqrt{d}$, $2/\sqrt{d}$, $1$ or $2$ modulo $\OO_{F}^{\times}$.
\newline

If $\div(I)\equiv 1/\sqrt{d}$, then $A_L$ is isomorphic to 
$a\cdot\delta\oplus\eta$
 as  $\OO_{F}$-modules.
\newline

If $\div(I)\equiv 2/\sqrt{d}$, the candidates are 
\[
a\cdot\delta\oplus\delta^c,\quad
a\cdot\delta\oplus 2\cdot\delta^2,\quad
a\cdot\delta\oplus\eta^2,\quad
a\cdot\delta\oplus 2\cdot\eta\]
where $c=0,2,3$.
\newline

If $\div(I)\equiv 1$, the candidates are 
\[
a\cdot\delta\oplus\eta^2,\quad
a\cdot\delta\oplus 2\cdot \eta.\]
 \newline

If $\div(I)\equiv 2$, the candidates are 
\[
a\cdot\delta\oplus\delta^c\oplus\eta^2,\quad
a\cdot\delta\oplus\delta^c\oplus 2\cdot\eta,\quad
a\cdot\delta\oplus 2\cdot\delta^2\oplus\eta^2,\quad
a\cdot\delta\oplus 2\cdot\delta^2\oplus 2\cdot\eta
\]
where $c=2,3$.

\subsection*{Acknowledgements}
The author would like to express his sincere gratitude to Shouhei Ma for helpful discussions and warm encouragement.
He also would like to thank Tetsushi Ito, his advisor, for his constructive suggestions, and Reimi Irokawa and Yuki Yamamoto for their assistance with the complicated calculations in the Appendix of this paper.
 Finally, he is deeply grateful to
the anonymous referees for helpful suggestions to make the exposition more
readable.

This work is supported by JST ACT-X JPMJAX200P.


\begin{thebibliography}{99}
  \bibitem{AMRT}
  Ash, A., Mumford, D., Rapoport, M., Tai, Y.,
  \textit{Smooth compactifications of locally symmetric varieties},
  Second edition.\ With the collaboration of Peter Scholze.\ Cambridge Mathematical Library.\ Cambridge University Press,\ Cambridge,\ 2010.
  
   \bibitem{Behrens}
  Behrens, N.,
  \textit{Singularities of ball quotients},
  Geom.\ Dedicata\ 159 (2012),\ 389-407. 

  \bibitem{infinite}
  Borcherds, R.,
  \textit{Automorphic forms on $\O_{s+2,2}(\R)$ and infinite products},
  Invent.\ Math.\ 120 (1995),\ no.\ 1,\ 161-213.
  
  \bibitem{DS}
  Diamond, F., Shurman, J.,
  \textit{A first course in modular forms},
  Graduate Texts in Mathematics, 228. Springer-Verlag, New York, 2005. xvi+436 pp.

  \bibitem{Freitag}
  Freitag, E.,
  \textit{Siegelsche Modulfunktionen},
  Grundlehren der Mathematischen Wissenschaften,\ Springer-Verlag,\ Berlin (1983).

  \bibitem{Reflective}
  Gritsenko, V.,
  \textit{Reflective modular forms in algebraic geometry},
  preprint (2010),\ arXiv:1012.4155.
 
  \bibitem{GH}
  Gritsenko, V., Hulek, K.,
  \textit{Moduli of polarized Enriques surfaces},
  in K3 surfaces and their moduli,\ 55-72,\ Progr.\ Math.,\ 315,\ 2016.

  \bibitem{GHS}
  Gritsenko, V., Hulek, K., Sankaran, G.K.,
  \textit{The Kodaira dimension of moduli of K3 surfaces},
  Invent.\ Math.\ 169 (2007),\ no.\ 3,\ 519-567.

  \bibitem{GHS4}
  Gritsenko, V., Hulek, K., Sankaran, G.K.,
  \textit{Moduli of K3 surfaces and irreducible symplectic manifolds},
  Handbook of moduli.\ Vol.\ I,\ 459-526,\ Adv.\ Lect.\ Math.\ (ALM),\ 24,\ Int.\ Press,\ Somerville,\ MA,\ 2013.
 
  \bibitem{HKN}
  Hentschel, M., Krieg, A, Nebe, G.,
  \textit{On the classification of lattices over $\Q(\sqrt{-3})$ which are even unimodular $\Z$-lattices},
  Abh.\ Math.\ Semin.\ Univ.\ Hambg.\ 80 (2010),\ no.\ 2.\  183-192.

  \bibitem{Hofmann}
  Hofmann, E.,
  \textit{Borcherds products on unitary groups},
  Math.\ Ann.\ 358 (2014),\ no.\ 3-4,\ 799-832.

  \bibitem{Kondo1}
  Kondo, S.,
  \textit{On the Kodaira dimension of the moduli spaces of K3 surfaces},
  Compositio.\ Math.\ 89 (1993),\ no.\ 3,\ 251-299.

  \bibitem{Kondo2}
  Kondo, S.,
  \textit{On the Kodaira dimension of the moduli spaces of K3 surfaces. II},
  Compositio.\ Math.\ 116 (1999),\ no.\ 2,\ 111-117.

  \bibitem{looijenga_ball}
  Looijenga, E.,
  \textit{Compactifications defined by arrangements. I. The ball quotient case},
  Duke Math. J. 118 (2003), no. 1, 151–187.

  \bibitem{Ma}
  Ma, S.,
  \textit{On the Kodaira dimension of orthogonal modular varieties},
  Invent.\ Math.\ 212 (2018),\ no.\ 3,\ 859-911.
  
  \bibitem{irregular}
  Ma, S.,
  \textit{Irregular cusps of orthogonal modular varieties},
  arXiv:2101.02950.
  
  \bibitem{boundary}
  Ma, S.,
  \textit{Boundary branch divisor of toroidal compactification},
  arXiv:2104.01933.
  
  \bibitem{Maeda2}
  Maeda, Y.,
  \textit{Uniruledness of unitary Shimura varieties associated with Hermitian forms of signatures $(1,3)$, $(1,4)$ and $(1,5)$},
  arXiv:2008.13106.
  
  \bibitem{MaedaOdaka}
  Maeda, Y., Odaka, Y.,
  \textit{Fano Shimura varieties with mostly branched cusps},
 to appear in the proceedings of the conference ``Birational geometry, Kahler-Einstein metrics and degenerations", 2022.
  
  \bibitem{Mumford}
  Mumford, D.,
  \textit{Hirzebruch's proportionality theorem in the noncompact case},
  Invent.\ Math.\ 42 (1977),\ 239-272.

  \bibitem{Tai}
  Tai, Y.,
  \textit{On the Kodaira dimension of the moduli space of abelian varieties},
  Invent.\ Math.\ 68 (1982),\ 425-439.
  
\bibitem{WW}
Wang, H., Williams, B.,
\textit{Free algebras of modular forms on ball quotients},
arXiv:2105.14892.


\end{thebibliography}
\end{document}